\documentclass[11pt,a4paper]{article}%
\usepackage{amssymb,amsmath,amsfonts,amsthm,array,bm,color}%
\usepackage{amsmath}%
\usepackage{amsfonts}%
\usepackage{amssymb}%
\usepackage{graphicx}
\providecommand{\U}[1]{\protect\rule{.1in}{.1in}}
\newtheorem{theorem}{Theorem}[section]
\newtheorem{lemma}[theorem]{Lemma}

\newtheorem{proposition}[theorem]{Proposition}
\newtheorem{remark}[theorem]{Remark}

\def\<{\langle}
\def\>{\rangle}
\def\d{{\rm d}}
\def\L{\mathcal{L}}
\def\div{{\rm div}}
\def\E{\mathbb{E}}
\def\N{\mathbb{N}}
\def\P{\mathbb{P}}
\def\R{\mathbb{R}}
\def\T{\mathbb{T}}
\def\Z{\mathbb{Z}}

\begin{document}

\title{Kolmogorov equations associated to the stochastic 2D Euler equations}

\author{Franco Flandoli\footnote{Email: franco.flandoli@sns.it. Scuola Normale Superiore of Pisa, Italy.} \ and
Dejun Luo\footnote{Email: luodj@amss.ac.cn. RCSDS, Academy of Mathematics and Systems Science, Chinese Academy of Sciences, Beijing 100190, China, and School of Mathematical Sciences, University of the Chinese Academy of Sciences, Beijing 100049, China. }}

\maketitle

\makeatletter
\renewcommand\theequation{\thesection.\arabic{equation}}
\@addtoreset{equation}{section} \makeatother

\begin{abstract}
The Kolmogorov equation associated to a stochastic 2D Euler equations with transport type noise and random initial conditions is studied by a direct approach, based on Fourier analysis, Galerkin approximation and Wiener chaos methods. The method allows us to generalize previous results and to understand the role of the regularity of the noise, in relation to a limiting value of roughness.
\end{abstract}



\section{Introduction}

Stochastic 2D Euler equations with transport type noise and random Gaussian initial conditions seem to us a rich subject for theoretical investigations and, with due idealization, a potentially interesting model for stationary inverse cascade turbulence. The topic was initiated by S. Albeverio and collaborators in a series of works and since then it received much attention; see for instance \cite{ARH, AH, AC}, the review \cite{AB} and references in \cite{FL}. The case with transport noise is more recent; we have initiated its investigation in \cite{FL}, where we have constructed solutions to the stochastic 2D Euler equations and have proved that their laws satisfy a certain Fokker--Planck equation and suitable gradient estimates. The present paper complements \cite{FL} with an entirely different approach: we study directly the associated Kolmogorov equation and prove existence and some regularity of solutions. The techniques are very different, based in \cite{FL} on point vortex approximation, here on Fourier analysis, Galerkin approximation and Wiener chaos methods. This approach allows us to extend the results of \cite{FL} and to prove a special property of the limit case $\gamma=2$, see below. We would like to mention that Kolmogorov equations in infinite dimensional spaces have been widely studied in the past, see e.g. \cite{KRZ, FG, DaPZ, BBDaPR, RS, Sauer, BKRS}; however, the equation treated in this paper is not covered by those ones, and it requires some special techniques.

Consider the vorticity formulation of the 2D stochastic Euler equation on $\T^2= \R^2/ \Z^2$:
  \begin{equation}\label{vorticity-Euler}
  \d \omega_t + u_t\cdot \nabla\omega_t +\sum_{k\in \Z^2_0} \sigma_k\cdot \nabla\omega_t \circ\d W^k_t=0,
  \end{equation}
where $\Z^2_0= \Z^2 \setminus \{0\}$, $\{W^k_\cdot \}_{k\in \Z^2_0}$ is a family of independent standard Brownian motions, and
  \begin{equation}\label{vector-fields}
  \sigma_k(x)=\frac{k^\perp}{|k|^\gamma} \begin{cases}
  \cos(2\pi k\cdot x), & k\in \Z^2_+, \\
  \sin(2\pi k\cdot x), & k\in \Z^2_-,
  \end{cases} \quad x\in \T^2,\, \gamma \geq 2,
  \end{equation}
with $k^\perp = (k_2,-k_1)$, $\Z^2_+ = \big\{k\in \Z^2_0: (k_1 >0) \mbox{ or } (k_1=0,\, k_2>0) \big\}$ and $\Z^2_- = -\Z^2_+$. The generator of $\omega_t$ is
  \begin{equation}\label{generator}
  \L F(\omega)= \frac12 \sum_{k\in \Z^2_0} \big\<\sigma_k\cdot \nabla\omega, D \<\sigma_k\cdot \nabla\omega, D F\>\big\> - \big\< u(\omega)\cdot \nabla\omega, D F\big\>, \quad F\in \mathcal{FC}_P,
  \end{equation}
where $\mathcal{FC}_P$ is the space of cylinder functionals on $H^{-1-}(\T^2)$ (see \eqref{cylinder-funct} below). Here, for $s\in \R$, we denote by $H^s(\T^2)$ the usual Sobolev spaces on $\T^2$ and $H^{-1-}(\T^2) = \cap_{s>0}\, H^{-1-s}(\T^2)$. Recall the Biot--Savart law:
  \begin{equation}\label{BS-law}
  u(\omega)(x)= \<\omega, K(x-\cdot)\> = \int_{\T^2} K(x-y)\,\omega(\d y),
  \end{equation}
where the Biot--Savart kernel $K$ has the expression
  \begin{equation}\label{BS-kernel}
  K(x)= 2\pi {\rm i} \sum_{k\in \Z^2_0} \frac{k^\perp}{|k|^2} {\rm e}^{2\pi {\rm i} k \cdot x} = -2\pi \sum_{k\in \Z^2_0} \frac{k^\perp}{|k|^2} \sin(2\pi k \cdot x).
  \end{equation}
Note that for all $p\in [1,2)$, $K\in L^p(\T^2,\d x)$, thus by \cite[p. 217, Theorem 3.5.7]{Loukas}, the above series converge to $K$ in $L^p(\T^2,\d x)$.

Let $\mu$ be the law of the white noise on $\T^2$, which is also called the enstrophy measure and supported by $H^{-1-}(\T^2)$. In the recent paper \cite{FL}, we proved that if $\gamma>2$ in \eqref{vector-fields}, then for any $\rho_0 \in C_b\big( H^{-1-}(\T^2), \R_+ \big)$ with $\int \rho_0 \,\d\mu =1$, the equation \eqref{vorticity-Euler} admits a white noise solution which is a stochastic process taking values in $H^{-1-}(\T^2)$, and the distribution at any time $t\in [0,T]$ has a density $\rho_t$ w.r.t. to $\mu$. Moreover, $\rho_t$ satisfies the forward Kolmogorov equation (or Fokker--Planck equation)
  \begin{equation}\label{Kolmog-eq}
  \partial_t \rho_t = \L^\ast \rho_t, \quad \rho|_{t=0} = \rho_0
  \end{equation}
associated to the operator $\L$ in \eqref{generator}, and the following gradient estimate holds:
  $$\sum_{k\in \Z^2_0} \int_0^T \!\! \int_{H^{-1-}(\T^2)} \<\sigma_k\cdot \nabla\omega, D \rho_t\>^2 \,\d\mu \d t\leq \|\rho_0\|_\infty^2. $$
The method in \cite{FL} is based on the idea that the enstrophy measure $\mu$ can be approximated by point vortices when the number of points goes to infinity.

The purpose of the current work is to solve the forward Kolmogorov equation \eqref{Kolmog-eq}, by using the method of Galerkin approximation. We can prove the same results as in \cite{FL} for the case $\gamma >2$ (see Theorem \ref{main-result} below), provided that $\rho_0\in L^2 \big( H^{-1-}(\T^2), \mu \big)$ instead of $\rho_0 \in C_b\big( H^{-1-}(\T^2) \big)$. However, when $\gamma =2$, it turns out that any weak limit of the Galerkin approximation is trivial. We remark that, since $\sigma_k \cdot \nabla \omega$ ($\forall \, k\in \Z_0^2$) and $u(\omega) \cdot \nabla \omega$ can be viewed as divergence free fields w.r.t. the white noise measure $\mu$ on $H^{-1-}(\T^2)$, the operators $\L$ and $\L^\ast$ differ from each other only by a sign of the drift parts. Therefore the approach of this paper works also for the backward Kolmogorov equation with little change. In order to state precisely our main results, we introduce some notations.

First, we use $\<\cdot, \cdot\>$ to denote the dual pairing between the space $C^\infty(\T^2)'$ of distributions and the test functions $C^\infty(\T^2)$. Let $\{e_k: k\in \Z^2_0\}$ be the usual orthonormal basis of $L^2(\T^2,\R)$ consisting of zero-average functions:
  \begin{equation}\label{real-basis}
  e_k(x) = \sqrt{2} \begin{cases}
  \cos(2\pi k\cdot x), & k\in \Z^2_+ , \\
  \sin(2\pi k\cdot x), & k\in \Z^2_-.
  \end{cases}
  \end{equation}
By $\Lambda \Subset \Z^2_0$ we mean that $\Lambda$ is a finite subset of $\Z^2_0$, and $\R^\Lambda$ is the $(\#\Lambda)$-dimensional Euclidean space. We introduce the family of cylindrical functions on $H^{-1-}$:
  \begin{equation}\label{cylinder-funct}
  \mathcal{FC}_P:= \big\{F(\omega)= f(\<\omega, e_l\>; l\in \Lambda) \mbox{ for some } \Lambda\Subset \Z^2_0 \mbox{ and } f\in C_P^\infty \big(\R^\Lambda \big) \big\},
  \end{equation}
where $C_P^\infty\big(\R^\Lambda \big) $ is the collection of smooth functions on $\R^\Lambda$ having polynomial growth together with all the derivatives. To simplify notations, we shall write $F(\omega) = f\circ \Pi_\Lambda(\omega)$, see Section \ref{sec-app-operator} for details. For a cylindrical function $F = f\circ \Pi_\Lambda$, we define
  $$DF= D F(\omega)= \sum_{j\in \Lambda} ((\partial_j f) \circ \Pi_\Lambda)\, e_j,$$
where $\partial_j f= \partial_{\xi_j} f$ is the partial derivative.

We regard $\sigma_k \cdot \nabla\omega$ as a distribution which is understood as follows: for any $\phi\in C^\infty(\T^2)$,
  $$\<\sigma_k\cdot \nabla\omega, \phi\> = - \<\omega, \sigma_k\cdot \nabla\phi\>,$$
since $\sigma_k$ is smooth and divergence free on $\T^2$. Then for any cylindrical function $F= f\circ \Pi_\Lambda$, we have
  $$\<\sigma_k\cdot \nabla\omega, D F\> = \sum_{j\in \Lambda} ((\partial_j f) \circ \Pi_\Lambda) \<\sigma_k\cdot \nabla\omega, e_j\>. $$
Note that $\<\sigma_k\cdot \nabla\omega, D F\>$ is also a cylindrical function.\footnote{This is the reason why we require that, in the definition of cylindrical functions, $f\in C_P^\infty \big(\R^\Lambda \big)$ instead of $f\in C_b^\infty \big(\R^\Lambda \big)$, since $\<\sigma_k\cdot \nabla\omega, e_j\>$ is unbounded.} It can be shown that $\div_\mu(\sigma_k\cdot \nabla\omega)=0$ in the distributional sense (see \cite[Lemma 4.5]{FL}), i.e., for any cylindrical function $F$,
  $$\int \<\sigma_k\cdot \nabla\omega, D F\> \,\d\mu =0.$$
The meaning of the drift part $\< u(\omega)\cdot \nabla\omega, D F \>$ in \eqref{generator} is more delicate, and the reader is referred to Section \ref{appendix-drift} of this paper (see also \cite{DaFR}).

Given $G\in L^2(H^{-1-},\mu)$, we say that $\<\sigma_k\cdot \nabla\omega, D G\>$ exists in the distributional sense if there exists some $\varphi \in L^2(H^{-1-},\mu)$ such that, for any cylindrical function $F$, it holds
  $$\int \<\sigma_k\cdot \nabla\omega, D F\> G \,\d\mu = - \int F \varphi \,\d\mu. $$
In this case we shall write $\<\sigma_k\cdot \nabla\omega, D G\> =\varphi$. We can give a similar definition when $G$ is time-dependent.

Using the method of Galerkin approximation (see Section \ref{sec-galerkin}), we can prove

\begin{theorem} \label{main-result}
Assume $\rho_0\in L^2(H^{-1-},\mu)$ and $\gamma\geq 2$ in \eqref{vector-fields}. There exists a measurable function $\rho \in L^\infty \big(0,T; L^2( H^{-1-}, \mu) \big)$ such that
\begin{itemize}
\item[\rm (i)] for every $k\in \Z^2_0$, $\<\sigma_k \cdot \nabla\omega , D \rho_t \>$ exists in the distributional sense;
\item[\rm (ii)] the gradient estimate holds:
  \begin{equation}\label{main-result.1}
  \sum_{k\in \Z^2_0} \int_0^T \!\! \int \<\sigma_k \cdot \nabla\omega , D \rho_t \>^2 \, \d\mu\d t \leq \|\rho_0 \|_{L^2(\mu)}^2;
  \end{equation}
\item[\rm (iii)] if $\gamma> 2$, then for any cylindrical function $F$ and $\alpha \in C^1 \big([0,T],\R \big)$ satisfying $\alpha(T)=0$, one has
  \begin{equation}\label{main-result.2}
  \aligned
  0= &\ \alpha(0)\int F \rho_0 \,\d\mu + \int_0^T \!\! \int \rho_t \big( \alpha'(t) F - \alpha(t) \< u(\omega) \cdot\nabla\omega , D F\> \big) \,\d\mu\d t \\
  & - \frac12 \sum_{k\in \Z^2_0} \int_0^T \!\! \int \alpha(t) \< \sigma_k \cdot \nabla \omega, D F\> \< \sigma_k \cdot \nabla\omega, D \rho_t \> \,\d\mu\d t.
  \endaligned
  \end{equation}
\end{itemize}
\end{theorem}

\begin{remark}
Any solution $\rho_t$ to the equation \eqref{main-result.2} is weakly continuous in $t$, namely, for any cylindrical function $F$, $t\to \int \rho_t F \,\d\mu$ is continuous on $[0,T]$. Indeed, we deduce from \eqref{main-result.2} that, in the distributional sense,
  $$\frac{\d}{\d t} \int \rho_t F \,\d\mu = \int \rho_t \< u(\omega) \cdot\nabla\omega , D F\> \,\d\mu - \frac12 \sum_{k\in \Z^2_0} \int \< \sigma_k \cdot \nabla \omega, D F \> \< \sigma_k \cdot \nabla\omega, D \rho_t \> \,\d\mu.$$
Since the r.h.s. is integrable on $[0,T]$, we conclude that $t\to \int \rho_t F \,\d\mu$ is absolutely continuous. In particular, taking $F \equiv 1$, we obtain $\int \rho_t\,\d\mu = \int \rho_0\,\d\mu$ for all $t\in [0,T]$.
\end{remark}

The case $\gamma=2$ is quite tricky and we are unable to show that the limit satisfies some equation. Indeed, by Proposition \ref{prop-gradient-cylinder} below, for any nonconstant cylindrical function $F$,
  $$ \sum_{k\in \Z^2_0} \int \< \sigma_k\cdot \nabla \omega, D F \>^2 \,\d\mu <+\infty$$
if and only if $\gamma>2$ in \eqref{vector-fields}. This result suggests that, when $\gamma=2$, the diffusion part of $\L F(\omega)$ in \eqref{generator} might be divergent. We can prove this claim by decomposing the partial sum of the diffusion part into two terms: the first one is convergent in any $L^p(H^{-1-},\mu)$, while the second one explodes at a logarithmic rate (see Proposition \ref{prop-operator}). Based on this decomposition, for a slightly modified approximation scheme (see \eqref{approx-generator-1} for details), we can show that the projections of the limit $\rho_t$ on any nonconstant Hermite polynomials vanish.

\begin{theorem}\label{main-result-2}
Assume $\gamma=2$. Then any limit points of the modified Galerkin approximation is trivial, i.e., for all $t\in (0,T]$, $\rho_t = \int \rho_0\,\d\mu$ a.e.
\end{theorem}

This paper is organised as follows. In Section 2, we make some preparations concerning the Galerkin approximation for the equations \eqref{vorticity-Euler} and \eqref{Kolmog-eq}, and recall some basic facts of the Hermite polynomials on $H^{-1-}(\T^2)$. Based on these results, we prove the three assertions of Theorem \ref{main-result} in Section 3. Theorem \ref{main-result-2} will be proved in Section 4, thanks to a decomposition of the approximation operator. In the appendices, we show the convergence of the Galerkin approximation for the nonlinear part in \eqref{generator}, as well as the convergence of one part in the decomposition of the approximation operator, which is similar to a renormalization argument.

\section{Some preparations} \label{sec-galerkin}

\subsection{The finite dimensional approximation of $\L$}\label{sec-app-operator}

Let $\Lambda\subset \Z^2_0$ be a finite subset. We denote by $H_\Lambda = {\rm span} \{e_k: k\in \Lambda\} \subset L^2(\T^2,\d x)$ and define the projection operator $\Pi_\Lambda: L^2(\T^2,\d x) \to H_\Lambda$ as
  $$\Pi_\Lambda f = \sum_{k\in \Lambda} \<f, e_k\> e_k,\quad f\in L^2(\T^2,\d x),$$
where $\<\cdot, \cdot\>$ is the inner product of $L^2(\T^2,\d x)$. $H_\Lambda$ is isomorphic to the Euclidean space $\R^\Lambda$, and $\Pi_\Lambda$ can be extended to the whole $H^{-1-}(\T^2)$:
  \begin{equation}\label{projection-1}
  \Pi_\Lambda \omega = \sum_{k\in \Lambda} \<\omega, e_k\> e_k, \quad \omega\in H^{-1-}(\T^2),
  \end{equation}
in which $\<\cdot, \cdot\>$ is now the duality between distributions and smooth functions. Let $\Lambda_N = \{k\in \Z^2_0: |k_1|\vee |k_2|\leq N\}$ for some $N\in \N$, we simply write $H_N= H_{\Lambda_N}$ and $\Pi_N= \Pi_{\Lambda_N}$. Using these notations, we can rewrite the cylindrical functions \eqref{cylinder-funct} as
  \begin{equation}\label{cylinder-funct-1}
  \mathcal{FC}_P:= \big\{F: H^{-1-} \to \R\ |\, \exists\, \Lambda\Subset \Z^2_0 \mbox{ and } f\in C_P^\infty (H_\Lambda) \mbox{ s.t. } F= f\circ\Pi_\Lambda\big\}.
  \end{equation}

Although we work in the framework of real-valued functions, it is sometimes easier to do computations by using the complex basis below:
  $$\tilde e_k(x)= {\rm e}^{2\pi {\rm i} k \cdot x},\quad x\in \T^2,\, k\in \Z^2_0.$$
The family $\{\tilde e_k: k\in \Z^2_0\}$ is a complete orthonormal system of $H^{\mathbb C}= L^2_0(\T^2, \mathbb C)$, the space of square integrable functions with zero average. The following identities are very useful:
  \begin{equation}\label{properties-basis}
  \tilde e_k(x) \tilde e_l(x)= \tilde e_{k+l}(x),\quad \overline{\tilde e_k(x)}= \tilde e_{-k}(x), \quad (\tilde e_k\ast \tilde e_l)(x) = \delta_{k,l}\, \tilde e_k(x),
  \end{equation}
where $\ast$ means the convolution. Moreover, we have the well known relations below:
  \begin{equation}\label{real-basis.1}
  e_k(x) = \begin{cases}
  \frac{\sqrt 2}2 (\tilde e_k(x) + \tilde e_{-k}(x)), & k\in \Z^2_+, \\
  \frac{\sqrt 2}{2i} (\tilde e_k(x) - \tilde e_{-k}(x)), & k\in \Z^2_-.
  \end{cases}
  \end{equation}

For $N\in \N$, set
  $$H_N^{\mathbb C} = {\rm span} \{\tilde e_k: k\in \Lambda_N\}, $$
which is a subspace of $H^{\mathbb C}$. The space $H_N$ defined above is the subspace of $H_N^{\mathbb C}$ consisting of real-valued elements: $\xi = \sum_{k\in \Lambda_N} \xi_k \tilde e_k \in H_N$ if and only if $\overline{\xi_k} = \xi_{-k}$ for all $k\in \Lambda_N$. Note that
  \begin{equation}\label{sec-2.1}
  \overline{\<\omega, \tilde e_k\>} = \big\<\omega, \overline{\tilde e_k}\, \big\> = \<\omega, \tilde e_{-k}\>,
  \end{equation}
thus it is not difficult to show that the projection $\Pi_N = \Pi_{\Lambda_N}$ can also be written as
  \begin{equation}\label{sec-2.2}
  \Pi_N \omega = \sum_{k\in \Lambda_N} \<\omega, \tilde e_k\> \tilde e_k.
  \end{equation}

Now we project the drift term $u(\omega)\cdot \nabla\omega$ in \eqref{vorticity-Euler} as follows:
  $$b_N(\omega):=\Pi_N \big(u(\Pi_N \omega) \cdot \nabla (\Pi_N \omega) \big),\quad \omega\in H^{-1-},$$
where $u(\Pi_N \omega)$ is obtained by replacing $\omega$ in \eqref{BS-law} with $\Pi_N \omega$. We shall consider $b_N$ as a vector field on $H_N$ whose generic element is denoted by $\xi= \sum_{k\in \Lambda_N} \xi_k \tilde e_k$. Thus
  $$b_N(\xi)= \Pi_N \big(u(\xi) \cdot \nabla \xi\big), \quad \xi\in H_N.$$
Analogously, we define the projection of the diffusion coefficient $\sigma_k\cdot \nabla \omega$ in \eqref{vorticity-Euler}:
  $$G_N^k(\xi)= \Pi_N\big(\sigma_k\cdot \nabla \xi\big), \quad \xi\in H_N.$$

In order to compute the expression of $b_N$ and $G_N^k$, we need the simple fact
  \begin{equation}\label{thm-app.4}
  \tilde e_k\ast K = 2\pi {\rm i} \frac{k^\perp}{|k|^2} \tilde e_k,
  \end{equation}
where $K$ is the Biot--Savart kernel. Indeed, let $K_N= 2\pi {\rm i} \sum_{k\in \Lambda_N} \frac{k^\perp}{|k|^2} \tilde e_k$ be the approximation of $K$. Since $K$ is integrable, by \cite[Page 217, Theorem 3.5.7]{Loukas}, its Fourier series converge to itself in $L^1(\T^2, \d x)$. Hence by the last equality in \eqref{properties-basis},
  $$\tilde e_k\ast K = \lim_{N\to \infty} \tilde e_k \ast K_N = 2\pi {\rm i} \frac{k^\perp}{|k|^2} \tilde e_k.$$

Recall that for $\xi= \sum_{k\in \Lambda_N} \xi_k \tilde e_k \in H_N$, one has $\overline{\xi_k} = \xi_{-k}$ for all $k\in \Lambda_N$. First, by the Biot--Savart law \eqref{BS-law} and \eqref{thm-app.4},
  \begin{equation}\label{lem-1.1}
  u(\xi)(x)= \int_{\T^2} K(x-y) \bigg(\sum_{l\in \Lambda_N} \xi_l \tilde e_l(y) \bigg) \d y = 2\pi {\rm i} \sum_{l\in \Lambda_N} \xi_l \frac{l^\perp}{|l|^2} \tilde e_l(x).
  \end{equation}
Next,
  \begin{equation}\label{eq-1}
  \nabla\xi(x)= \sum_{k\in \Lambda_N} \xi_k \nabla \tilde e_k(x)= 2\pi {\rm i} \sum_{k\in \Lambda_N} \xi_k \tilde e_k(x) k.
  \end{equation}
Combining the above two identities yields
  $$\big(u(\xi) \cdot \nabla \xi\big)(x)= -4\pi^2 \sum_{k,l\in \Lambda_N} \xi_k \xi_l \frac{k\cdot l^\perp}{|l|^2} \tilde e_{k+l}(x) .$$
It is easy to check that $\big(u(\xi) \cdot \nabla \xi\big)(x)$ is real-valued. Moreover,
  \begin{equation}\label{drift-approx}
  \aligned
  b_N(\xi) &= -4\pi^2 \sum_{k,l\in \Lambda_N} \textbf{1}_{\Lambda_N}(k+l) \xi_k \xi_l \frac{k\cdot l^\perp}{|l|^2} \tilde e_{k+l} \\
  &= -4\pi^2 \sum_{j\in \Lambda_N} \bigg[\sum_{l\in \Lambda_N} \textbf{1}_{\Lambda_N}(j-l) \xi_l \xi_{j-l} \frac{j\cdot l^\perp}{|l|^2} \bigg] \tilde e_j ,
  \endaligned
  \end{equation}
where $\textbf{1}_{\Lambda_N}$ is the indicator function. For $j\in \Lambda_N$, the $j$-th component of $b_N(\xi)$ is
  \begin{equation}\label{drift-approx.1}
  (b_N(\xi))_j= -4\pi^2 \sum_{l\in \Lambda_N} \textbf{1}_{\Lambda_N}(j-l) \xi_l \xi_{j-l} \frac{j\cdot l^\perp}{|l|^2} .
  \end{equation}

We turn to the diffusion coefficients $G_N^k$. By \eqref{vector-fields} and \eqref{real-basis.1}, we have, for $k\in \Z^2_+$,
  \begin{equation}\label{diffusion.1}
  \sigma_k(x) = \frac{k^\perp}{\sqrt{2}\, |k|^\gamma} (\tilde e_k(x) + \tilde e_{-k}(x)), \quad x\in \T^2.
  \end{equation}
Combining this with \eqref{eq-1} leads to
  $$(\sigma_k\cdot \nabla \xi)(x)= \sqrt{2} \pi {\rm i} \sum_{l\in \Lambda_N} \xi_l \frac{k^\perp \cdot l}{|k|^\gamma} ( \tilde e_{k+l}(x) + \tilde e_{-k+l}(x)).$$
Therefore, for $k\in \Z^2_+$,
  \begin{equation}\label{diffusion-approx}
  \aligned
  G_N^k(\xi) &= \sqrt{2} \pi {\rm i} \sum_{l\in \Lambda_N} \xi_l \frac{k^\perp \cdot l}{|k|^\gamma} \big[ {\bf 1}_{\Lambda_N}(k+l) \tilde e_{k+l} + {\bf 1}_{\Lambda_N}(-k+l) \tilde e_{-k+l} \big] \\
  & = \sqrt{2} \pi {\rm i} \sum_{j\in \Lambda_N} \frac{k^\perp \cdot j}{|k|^\gamma} \big[ {\bf 1}_{\Lambda_N}(j-k) \xi_{j-k} + {\bf 1}_{\Lambda_N}(j+k) \xi_{j+k} \big] \tilde e_j.
  \endaligned
  \end{equation}
Similarly, for $k\in \Z^2_-$,
  \begin{equation}\label{diffusion-approx.1}
  \aligned
  G_N^k(\xi) &= \sqrt{2} \pi \sum_{l\in \Lambda_N} \xi_l \frac{k^\perp \cdot l}{|k|^\gamma} \big[ {\bf 1}_{\Lambda_N}(k+l) \tilde e_{k+l} - {\bf 1}_{\Lambda_N}(-k+l) \tilde e_{-k+l} \big] \\
  & = \sqrt{2} \pi \sum_{j\in \Lambda_N} \frac{k^\perp \cdot j}{|k|^\gamma} \big[ {\bf 1}_{\Lambda_N}(j-k) \xi_{j-k} - {\bf 1}_{\Lambda_N}(j+k) \xi_{j+k} \big] \tilde e_j.
  \endaligned
  \end{equation}
From the above two equalities we see that $G_N^k(\xi)$ is linear in $\xi$.

Before concluding this part, we introduce the finite dimensional approximation $\L_N$ of the operator $\L$ given in \eqref{generator}:
  \begin{equation}\label{approx-generator}
  \L_N \phi(\xi)= \frac12 \sum_{|k|\leq N} \Big\<G_N^k, \nabla_N \big\<G_N^k, \nabla_N \phi\big\>_{H_N} \Big\>_{\! H_N}(\xi) - \<b_N, \nabla_N \phi\>_{H_N}(\xi),
  \end{equation}
where $\nabla_N$ is the gradient operator on $H_N$, and the inner product $\<\cdot, \cdot\>_{H_N}$ is induced from $L^2(\T^2,\d x)$.

\subsection{The stochastic flow on $H_N$}\label{sec-flow}

Let $\big\{(W^k_t)_{t\geq 0}: k\in \Z^2_0 \big\}$ be a family of independent standard Brownian motions defined on some filtered probability space $\big(\Theta, \mathcal F, (\mathcal F_t)_{t\geq 0}, \P \big)$. We consider the following SDE on $H_N$:
  \begin{equation}\label{finite-dim-SDE}
  \d \omega^N_t = b_N\big(\omega^N_t \big)\,\d t + \sum_{|k| \leq N} G_N^k \big(\omega^N_t \big)\circ \d W^k_t,\quad \omega^N_0 = \xi \in H_N.
  \end{equation}
Note that the generator of $\omega^N_t$ is the formal adjoint operator of $\L_N$ defined in \eqref{approx-generator} (since $b_N$ is divergence free, cf. Proposition \ref{lem-divergence} below). All the results in this section remain valid if we replace $b_N$ by $-b_N$ in the equation \eqref{finite-dim-SDE}. Here we consider this form for the sake of application in the next subsection.

Since the vector fields $b_N$ and $G_N^k$ are smooth (see their expressions \eqref{drift-approx} and \eqref{diffusion-approx}), the equation \eqref{finite-dim-SDE} has a unique strong solution up to some life time $\zeta$. We will show that the solution is non-explosive. To this end, we first prove

\begin{lemma}\label{lem-1}
It holds that
  $$\<b_N(\xi), \xi\>_{H_N}= \big\<G_N^k(\xi), \xi \big\>_{H_N} =0 \quad \mbox{for all } \xi\in H_N,\, k\in \Lambda_N.$$
\end{lemma}

\begin{proof}
By \eqref{lem-1.1}, $u(\xi)$ is a smooth divergence free vector field on $\T^2$. We have
  $$\<b_N(\xi), \xi\>_{H_N} = \big\<\Pi_N \big(u(\xi) \cdot \nabla \xi\big), \xi\big\>_{H_N} = \big\<u(\xi) \cdot \nabla \xi, \xi\big\>_{H_N} = - \big\<\xi, u(\xi) \cdot \nabla \xi\big\>_{H_N}, $$
where the last step is due to integration by parts. This implies $\<b_N(\xi), \xi\>_{H_N} =0$. Similarly we obtain the second assertion since $\sigma_k$ is divergence free.
\end{proof}

The next result shows that the trajectories of $\omega^N_t$ stay on the sphere centered at the origin.

\begin{lemma}\label{lem-energy}
A.s., for all $t>0$,
  \begin{equation}\label{energy-conservation}
  \big|\omega^N_t \big|_{H_N} = |\xi |_{H_N}.
  \end{equation}
\end{lemma}

\begin{proof}
By Lemma \ref{lem-1}, it is clear that
  $$\L_N^\ast \big(|\cdot|_{H_N}^2 \big)(\xi)=0 \quad \mbox{for all } \xi \in H_N.$$
Let $R>0$ be fixed; define the stopping time $\tau_R = \inf\big\{t>0: \big|\omega^N_t\big|_{H_N} >R \big\}$. The It\^o formula and Lemma \ref{lem-1} lead to
  $$\d \big(\big|\omega^N_{t\wedge \tau_R} \big|_{H_N}^2 \big)= \sum_{|k| \leq N} \big\<G_N^k, \nabla_N \big(|\cdot|_{H_N}^2 \big) \big\> \big(\omega^N_{t\wedge \tau_R} \big) \,\d W^k_t + \L_N^\ast \big(|\cdot|_{H_N}^2 \big) \big(\omega^N_{t\wedge \tau_R} \big) \,\d t =0.$$
Therefore, for any $t>0$ and $R>0$.
  $$\big|\omega^N_{t\wedge \tau_R} \big|_{H_N}^2 = |\xi|_{H_N}^2 \quad \mbox{a.s.}$$
Letting $R\to \infty$ yields that, for any $t>0$,
  $$\big|\omega^N_t \big|_{H_N} = |\xi|_{H_N} \quad \mbox{a.s.}$$
Since $\omega^N_t$ has continuous trajectories, we obtain the desired assertion.
\end{proof}

\begin{proposition}\label{prop-flow}
The SDE \eqref{finite-dim-SDE} generates a stochastic flow $\big(\Phi^N_t \big)_{t\geq 0}$ of diffeomorphisms which preserve the norm of $H_N$.
\end{proposition}

\begin{proof}
By Lemma \ref{lem-energy}, for any $\xi\in H_N$, \eqref{finite-dim-SDE} has a unique strong solution $\omega^N_t$ which stays on the sphere $\{\eta\in H_N: |\eta|_{H_N}= |\xi|_{H_N}\}$. We denote it by $\Phi^N_t(\xi)$.

For any $n\in \N$, we take $b_N^{(n)}, G_N^{k, (n)} \in C_b^\infty(H_N, H_N)$ such that
  \begin{equation}\label{prop-flow.1}
  b_N^{(n)}(\xi) = b_N(\xi), \quad G_N^{k, (n)}(\xi) = G_N^k(\xi), \quad \xi \in B_n(H_N),
  \end{equation}
where $B_n(H_N)$ is the ball in $H_N$ with radius $n$ and centered at the origin. Consider the SDE
  \begin{equation}\label{prop-flow.1.5}
  \d \eta^{(n)}_t = b_N^{(n)}\big(\eta^{(n)}_t \big)\,\d t + \sum_{|k|\leq N} G_N^{k, (n)} \big(\eta^{(n)}_t \big)\circ \d W^k_t,\quad \eta^{(n)}_0 = \xi \in H_N.
  \end{equation}
It is well known that the above equation generates a stochastic flow of diffeomorphisms $\Phi^{(n)}_t$ on $H_N$. That is, there exists a measurable set $\Theta_n \subset \Theta$ of probability 1, such that for all $\theta\in \Theta_n$, $\Phi^{(n)}_t(\cdot, \theta)$ is a diffeomorphism on $H_N$ for all $t>0$. We denote by $\Theta_\infty = \cap_{n\geq 1} \Theta_n$ which is again a full measure set.

By the pathwise uniqueness of \eqref{finite-dim-SDE}, for all $\xi \in B_n(H_N)$, and for all $m>n$,
  \begin{equation}\label{prop-flow.2}
  \Phi^N_t(\xi) = \Phi^{(m)}_t(\xi) \quad \mbox{a.s. for all } t>0.
  \end{equation}
Therefore, we redefine the stochastic flow $\Phi^N_t$ as follows: for all $\theta \in \Theta_\infty$ and $n\geq 1$,
  $$\Phi^N_t(\cdot,\theta)|_{B_n(H_N)} = \Phi^{(n)}_t(\cdot,\theta)|_{B_n(H_N)}, \quad t>0.$$
By \eqref{prop-flow.2}, this definition is consistent and it gives us the desired flow of diffeomorphisms.
\end{proof}

Recall that $\mu$ is the enstrophy measure which is supported by $H^{-1-}(\T^2)$. Denote by $\mu_N= (\Pi_N)_\# \mu = \mu \circ \Pi_N^{-1}$ the induced standard Gaussian measure on $H_N$.

\begin{proposition}\label{lem-divergence}
A.s., for all $t>0$, the standard Gaussian measure $\mu_N$ is invariant under the stochastic flow $\Phi^N_t$ of diffeomorphisms on $H_N$.
\end{proposition}

\begin{proof}
By \cite[(2.2)]{FLT} (with a truncation argument as in \eqref{prop-flow.1}), it is sufficient to show that
  \begin{equation}\label{lem-divergence.1}
  \div_{\mu_N}(b_N) = \div_{\mu_N}\big(G_N^k \big) =0,
  \end{equation}
where $\div_{\mu_N}$ is the divergence operator on $H_N$ w.r.t. the Gaussian measure $\mu_N$.  Note that
  $$\div_{\mu_N}(b_N)(\xi) = \<b_N(\xi), \xi\>_{H_N}- \div_N(b_N)(\xi),$$
where $\div_N$ is the ordinary divergence operator on $H_N$. By Lemma \ref{lem-1}, it is sufficient to prove that $\div_N(b_N)(\xi) =0$ for all $\xi\in H_N$. The latter is obvious from the expression \eqref{drift-approx.1} of $(b_N(\xi))_j$, which does not contain  $\xi_j$. Similarly, we have $\div_{\mu_N} \big(G_N^k \big)(\xi) =0$.
\end{proof}

\subsection{Forward Kolmogorov equation on $H_N$}\label{sec-finite-Kol-eq}

We consider the Kolmogorov equation on $H_N$:
  \begin{equation}\label{Kolmogorov-eq}
  \partial_t \rho^N_t = \L_N^\ast \rho^N_t, \quad \rho^N|_{t=0}= \rho^N_0\in C_P^\infty(H_N),
  \end{equation}
where $\L_N^\ast$ is the formal adjoint operator of that defined in \eqref{approx-generator}, and it is the generator associated to the SDE \eqref{finite-dim-SDE}. Recall that $C_P^\infty(H_N)$ is the collection of smooth functions having polynomial growth together with all the derivatives. We have the following simple result.

\begin{lemma}\label{lem-probab-repres}
The equation \eqref{Kolmogorov-eq} has a smooth solution $\big(\rho^N_t \big)_{0\leq t\leq T}$ with the probabilistic representation
  \begin{equation}\label{probab-repres}
  \rho^N_t (\xi) = \E\big[ \rho^N_0 \big(\omega^N_t \big)\big]= \E\big[ \rho^N_0 \big(\Phi^N_t(\xi) \big)\big], \quad \xi\in H_N.
  \end{equation}
\end{lemma}

\begin{proof}
For any $n\geq 1$, consider the cut-off functions as in \eqref{prop-flow.1}. Moreover, let $\rho^{N,(n)}_0 \in C_b^\infty(H_N, \R)$ such that
  $$\rho^{N,(n)}_0(\xi)= \rho^{N}_0(\xi) \quad \mbox{for all } \xi \in B_n(H_N).$$
Define the new operator
  $$\L_N^{(n)} \phi(\xi)= \frac12 \sum_{|k|\leq N} \Big\<G_N^{k,(n)}, \nabla_N \big\<G_N^{k,(n)}, \nabla_N \phi\big\>_{H_N} \Big\>_{\! H_N}(\xi) - \big\< b_N^{(n)}, \nabla_N \phi \big\>_{H_N}(\xi)$$
and consider the equation
  \begin{equation}\label{lem-probab-repres.1}
  \partial_t \rho^{N,(n)}_t = \big(\L_N^{(n)} \big)^\ast \rho^{N,(n)}_t, \quad \rho^{N,(n)}|_{t=0}= \rho^{N,(n)}_0.
  \end{equation}
It is well known that this equation has a smooth solution $\big(\rho^{N,(n)}_t \big)_{0\leq t\leq T}$. Recall the solution $\eta^{(n)}_t = \Phi^{(n)}_t (\xi)$ to the cut-off equation \eqref{prop-flow.1.5}. By the It\^o formula, it is easy to show that
  $$\rho^{N,(n)}_t(\xi)= \E \big[ \rho^{N,(n)}_0 \big(\Phi^{(n)}_t (\xi)\big) \big],\quad (t,\xi)\in [0,T]\times H_N.$$

Fix any $n\in \N$. By \eqref{prop-flow.2}, for all $\xi\in B_n(H_N)$ and $m\geq n$,
  $$\rho^{N,(m)}_t(\xi)= \E \big[ \rho^{N,(m)}_0 \big(\Phi^{(m)}_t (\xi)\big) \big] = \E \big[ \rho^{N}_0 \big(\Phi^N_t (\xi)\big) \big],$$
which is independent on $m$. Therefore, by \eqref{lem-probab-repres.1}, the function $\rho^N_t (\xi)= \E \big[ \rho^{N}_0 \big(\Phi^N_t (\xi)\big) \big]$ verifies
  $$\partial_t \rho^N_t = \big(\L_N^{(n)} \big)^\ast \rho^{N,(n)}_t = \L_N^\ast \rho^{N}_t \quad \mbox{on } (t,\xi)\in [0,T]\times B_n(H_N). $$
Since $n$ is arbitrary, we conclude the result.
\end{proof}

We shall establish some estimates on the solution $\rho^N_t$ by using the above representation formula and the equation \eqref{Kolmogorov-eq}.

\begin{lemma}\label{lem-estimate}
For any $p\geq 1$,
  $$\big\|\rho^N \big\|_{L^\infty([0,T], L^p(\mu_N))} \leq \big\|\rho^N_0 \big\|_{L^p(\mu_N)}. $$
Moreover,
  \begin{equation}\label{estimate-2}
  \big\|\rho^N_t \big\|_{L^2(H_N)}^2 + \sum_{|k|\leq N} \int_0^t \!\! \int_{H_N} \big\<G_N^k, \nabla_N \rho^N_s \big\>^2 \,\d\mu_N \d s = \big\|\rho^N_0 \big\|_{L^2(H_N)}^2, \quad t\in [0,T].
  \end{equation}
\end{lemma}

\begin{proof}
By \eqref{probab-repres}, for any $t\in [0,T]$,
  $$ \int_{H_N} \big|\rho^N_t(\xi) \big|^p\,\mu_N(\d\xi) \leq \int_{H_N} \E \big|\rho^N_0 \big( \Phi^N_t(\xi) \big) \big|^p\,\mu_N(\d\xi)= \int_{H_N} \big|\rho^N_0 (\xi)\big|^p\,\mu_N(\d\xi), $$
where the last step follows from the fact that $\Phi^N_t$ preserves the Gaussian measure $\mu_N$.

Next, by \eqref{Kolmogorov-eq}, we have
  $$\frac{\d}{\d t} \big\|\rho^N_t \big\|_{L^2(H_N)}^2 = 2 \int_{H_N} \rho^N_t \partial_t \rho^N_t \,\d\mu_N = 2 \int_{H_N} \rho^N_t \L_N^\ast \rho^N_t \,\d\mu_N.$$
Using the expression \eqref{approx-generator} and integrating by parts yield
  $$ \aligned
  \frac{\d}{\d t} \big\|\rho^N_t \big\|_{L^2(H_N)}^2 &= \sum_{|k|\leq N} \int_{H_N} \! \rho^N_t \Big\<G_N^k, \nabla_N \big\<G_N^k, \nabla_N \rho^N_t\big\> \Big\> \,\d\mu_N  + 2 \int_{H_N} \! \rho^N_t \big\<b_N, \nabla_N \rho^N_t \big\>\,\d\mu_N \\
  & = - \sum_{|k|\leq N} \int_{H_N} \big\<G_N^k, \nabla_N \rho^N_t\big\>^2 \,\d\mu_N .
  \endaligned$$
Therefore, integrating from $0$ to $t$ gives us \eqref{estimate-2}.
\end{proof}

\subsection{Hermite polynomials on $H^{-1-}(\T^2)$}

In this part we give a short introduction of the Hermite polynomials on $H^{-1-}(\T^2)$, see the recent book \cite[Chap. 5]{LR} for more details. Recall the one dimensional Hermite polynomials $h_n(t), n\geq 0$ which is characterized by
  \begin{equation*}
  \int_\R f^{(n)}(t) \,\d\nu_1(t) = \int_\R f(t) h_n(t) \,\d\nu_1(t),\quad \forall\, f\in C_b^\infty(\R).
  \end{equation*}
Here $\nu_1$ is the standard normal distribution on $\R$. $h_n$ has the expression
  \begin{equation*}
  h_n(t) = (-1)^n {\rm e}^{t^2/2} \frac{\d^n}{\d t^n} \big({\rm e}^{-t^2/2} \big),
  \end{equation*}
and it holds that
  \begin{equation*}
  h_n(t) = t h_{n-1}(t) - h'_{n-1}(t).
  \end{equation*}
It is well known that Hermite polynomials are eigenfunctions of the one dimensional Ornstein--Uhlenbeck operator:
  $$Af(t) = f''(t) - tf'(t), \quad f\in C_P^\infty(\R).$$
Indeed,
  $$A h_n = -n h_n, \quad n\geq 0.$$

Now we shall construct Hermite polynomials on $H^{-1-}(\T^2)$. Introduce the notations
  \begin{equation*}
  \bm{n} = (n_k )_{k\in \Z^2_0}\in (\N\cup \{0\})^{\Z^2_0}, \quad |\bm{n}| = \sum_{k\in \Z^2_0}^\infty n_k.
  \end{equation*}
Define the index set by
  $$\bm{N}= \big\{\bm{n}\in (\N\cup \{0\})^{\Z^2_0}: |\bm{n}|<+\infty \big\}.$$
For $\bm{n}\in \bm{N}$, there are only finitely many nonzero coordinates. Define
  \begin{equation*}
  H_{\bm{n}}(\omega) = \prod_{k\in \Z^2_0} h_{n_k}(\<\omega, e_k\>),\quad \omega \in H^{-1-}(\T^2).
  \end{equation*}
Since $h_0 \equiv 1$, this is indeed a finite product and $H_{\bm{n}}\in \mathcal{FC}_P$. Note that $\mathcal{FC}_P$ is dense in $L^2(H^{-1-}, \mu)$, one can prove that

\begin{lemma}\label{lem-basis}
The family $\{H_{\bm{n}}: \bm{n}\in \bm{N} \}$ is an orthogonal basis of $L^2(H^{-1-}, \mu)$.
\end{lemma}

Finally, we recall that the $n$-th Wiener chaos is defined as
  $$\mathcal C_n= \overline{\mbox{span}\{H_{\bm{n}}: |\bm{n}| = n\}}^{L^2(\mu)}. $$

\section{Forward Kolmogorov equation on $H^{-1-}(\T^2)$}

This section is devoted to the proof of Theorem \ref{main-result} and consists of two parts. In Section 3.1, we prove the first two claims of Theorem \ref{main-result}. The last assertion is proved in Section 3.2.

\subsection{Existence of limit points and basic properties} \label{proof-first-main-result}

Now suppose that $\rho_0 \in L^2(H^{-1-},\mu)$. Then there exists a sequence $\big\{ \rho^N_0 \big\}_{N\in \N}$ such that $\rho^N_0 \in C_P^\infty (H_N)$ and
  \begin{equation}\label{projection}
  \lim_{N\to \infty} \big\|\rho^N_0\circ \Pi_N -\rho_0 \big\|_{L^2(\mu)} =0.
  \end{equation}
Let $\rho^N_t$ be the solution to \eqref{Kolmogorov-eq} with the initial condition $\rho^N_0$. Then the energy identity \eqref{estimate-2} holds. By a slight abuse of notations, we shall write
  \begin{equation*}
  \rho^N_t(\omega) = \rho^N_t(\Pi_N \omega), \quad (t,\omega) \in [0,T] \times H^{-1-}.
  \end{equation*}
Then $D \rho^N_t(\omega)= \big(\nabla_N \rho^N_t \big)(\Pi_N \omega) \in H_N$ for all $(t,\omega) \in [0,T] \times H^{-1-}$. Hence,
  \begin{equation}\label{eq-6-1}
  \aligned
  \big\<\sigma_k \cdot \nabla(\Pi_N \omega), D \rho^N_t(\omega) \big\>_{L^2(\T^2)} &= \big\<\Pi_N \big(\sigma_k \cdot \nabla(\Pi_N \omega) \big), \big(\nabla_N \rho^N_t \big)(\Pi_N \omega) \big\>_{H_N}\\
  &= \big\<G_N^k, \nabla_N \rho^N_t\big\>_{H_N} (\Pi_N \omega).
  \endaligned
  \end{equation}
Combining these facts with \eqref{estimate-2} yields that
  \begin{equation}\label{estimate-3}
  \big\|\rho^N_t \big\|_{L^2(\mu)}^2 + \sum_{|k|\leq N} \int_0^t \!\! \int_{H^{-1-}} \big\<\sigma_k \cdot \nabla(\Pi_N \omega), D \rho^N_s \big\>_{L^2(\T^2)}^2 \,\d\mu \d s = \big\|\rho^N_0 \big\|_{L^2(\mu)}^2, \quad t\in [0,T].
  \end{equation}
For $k\in \Z^2_0$ with $|k|>N$, we set $\big\<\sigma_k \cdot \nabla(\Pi_N \omega), D \rho^N_s \big\> \equiv 0$. Combining \eqref{projection} and \eqref{estimate-3}, we have proved

\begin{proposition}\label{prop-convergence}
\begin{itemize}
\item[\rm (1)] $\big\{\rho^N \big\}_{N\in \N}$ is a bounded sequence in $L^\infty \big([0,T], L^2(H^{-1-}, \mu) \big)$;
\item[\rm (2)] the family
  $$\big\{ \big\<\sigma_k \cdot \nabla(\Pi_N \omega), D \rho^N_t \big\>_{L^2(\T^2)}: (k, t, \omega) \in \Z^2_0 \times [0,T] \times H^{-1-} \big\}_{N \in \N}$$
is bounded in the Hilbert space $L^2\big(\Z^2_0 \times [0,T] \times H^{-1-}, \# \otimes \d t \otimes \mu\big)$, where $\#$ is the counting measure on $\Z^2_0$.
\end{itemize}
\end{proposition}

As a consequence, we obtain the following result which proves the first two assertions of Theorem \ref{main-result}.

\begin{theorem}\label{thm-1}
Assume $\rho_0\in L^2 \big( H^{-1-}, \mu \big)$ and $\gamma\geq 2$ in \eqref{vector-fields}. There exists a measurable function $\rho \in L^\infty \big( [0,T], L^2( H^{-1-}, \mu) \big)$ such that
\begin{itemize}
\item[\rm (i)] for every $k\in \Z^2_0$, $\big\<\sigma_k \cdot \nabla\omega , D \rho_t \big\>$ exists in the distributional sense;
\item[\rm (ii)] the gradient estimate holds:
  \begin{equation}\label{thm-1.0}
  \sum_{k\in \Z^2_0} \int_0^T \!\! \int \big\<\sigma_k \cdot \nabla\omega , D \rho_t \big\>^2 \, \d\mu\d t \leq \|\rho_0 \|_{L^2(\mu)}^2.
  \end{equation}
\end{itemize}
\end{theorem}

\begin{proof}
By Proposition \ref{prop-convergence}, there exists a subsequence $\{N_i\}_{i\in \N}$ such that
\begin{itemize}
\item[\rm (a)] $\rho^{N_i}$ converges weakly-$\ast$ to some $\rho$ in  $L^\infty \big([0,T], L^2(H^{-1-}, \mu) \big)$;
\item[\rm (b)] $\big\<\sigma_k \cdot \nabla(\Pi_{N_i} \omega), D \rho^{N_i}_t \big\>_{L^2(\T^2)}$ converges weakly to some $\varphi \in L^2\big(\Z^2_0 \times [0,T] \times H^{-1-}, \# \otimes \d t \otimes \mu\big)$.
\end{itemize}

Let $\alpha\in C([0,T], \R)$ and $\beta \in L^2\big(\Z^2_0 \times H^{-1-}, \# \otimes \mu\big)$ such that $\beta_k\in \mathcal{FC}_{P}$ for all $k\in \Z^2_0$.  By the assertion (b),
  $$\lim_{i\to \infty} \sum_{k\in \Z^2_0} \int_0^T \!\! \int \big\<\sigma_k \cdot \nabla(\Pi_{N_i} \omega), D \rho^{N_i}_t \big\>_{L^2(\T^2)} \alpha(t) \beta_k \,\d\mu \d t = \sum_{k\in \Z^2_0} \int_0^T \!\! \int \varphi_k (t)\alpha(t) \beta_k(t) \,\d\mu \d t.$$
Fix some $k\in \Z^2_0$, we assume that $\beta_j \equiv 0 $ for all $j\neq k$ and $\beta_k  = \beta_k\circ \Pi_\Lambda$ for some $\Lambda \Subset \Z_0^2$. Then the above limit reduces to
  \begin{equation}\label{thm-1.1}
  \lim_{i\to \infty} \int_0^T \!\! \int \big\<\sigma_k \cdot \nabla(\Pi_{N_i} \omega), D \rho^{N_i}_t \big\>_{L^2(\T^2)} \, \alpha(t) \beta_k \,\d\mu \d t = \int_0^T \!\! \int \varphi_k (t)\alpha(t) \beta_k \,\d\mu \d t.
  \end{equation}
For $N_i$ big enough, we have by \eqref{eq-6-1} and \eqref{lem-divergence.1} that
  $$\aligned
  &\hskip13pt \int_0^T \!\! \int \big\<\sigma_k \cdot \nabla(\Pi_{N_i} \omega), D \rho^{N_i}_t \big\>_{L^2(\T^2)} \, \alpha(t) \beta_k(\omega) \,\d\mu \d t\\
  &= \int_0^T \!\! \int_{H_{N_i}} \big\< G_{N_i}^k, \nabla_{N_i} \rho^{N_i}_t \big\>_{H_{N_i}} \! (\xi) \, \alpha(t) \beta_k(\xi) \,\d\mu_{N_i} \d t \\
  & = - \int_0^T \!\! \int_{H_{N_i}} \rho^{N_i}_t(\xi) \alpha(t) \big\< G_{N_i}^k, \nabla_{N_i} \beta_k \big\>_{H_{N_i}} \!(\xi) \,\d\mu_{N_i} \d t \\
  &= - \int_0^T \!\! \int \rho^{N_i}_t(\omega) \alpha(t) \big\< \sigma_k \cdot \nabla(\Pi_{N_i} \omega), D \beta_k \big\>_{L^2(\T^2)} \,\d\mu \d t.
  \endaligned$$
Lemma \ref{lem-3} below implies
  $$\aligned
  \int_0^T \!\! \int \big\<\sigma_k \cdot \nabla(\Pi_{N_i} \omega), D \rho^{N_i}_t \big\>_{L^2(\T^2)} \, \alpha(t) \beta_k \,\d\mu \d t &= - \int_0^T \!\! \int \rho^{N_i}_t \alpha(t)  \big\< \sigma_k \cdot \nabla\omega, D \beta_k \big\> \,\d\mu \d t \\
  & \to - \int_0^T \!\! \int \rho_t\, \alpha(t) \big\< \sigma_k \cdot \nabla\omega, D \beta_k \big\> \,\d\mu \d t,
  \endaligned$$
where the last step is due to (a). Combining this limit with \eqref{thm-1.1} yields
  $$\int_0^T \!\! \int \varphi_k(t) \alpha(t) \beta_k \,\d\mu \d t = - \int_0^T \!\! \int \rho_t\, \alpha(t) \big\< \sigma_k \cdot \nabla\omega, D \beta_k \big\> \,\d\mu \d t.$$
By the arbitrariness of $\alpha\in C([0,T])$ and $\beta_k \in \mathcal{FC}_P$, we see that, in the distributional sense,
  \begin{equation}\label{thm-1.2}
  \varphi_k (t) = \big\< \sigma_k \cdot \nabla\omega, D \rho_t\big\>.
  \end{equation}
We obtain the assertion (i). The second assertion follows from \eqref{projection}, \eqref{estimate-3} and the fact (b).
\end{proof}

\begin{lemma}\label{lem-3}
Let $k\in \Z^2_0$ and $G\in \mathcal{FC}_P$ be fixed. For  all $N$ big enough,
  $$\big\< \sigma_k \cdot \nabla(\Pi_N\omega), D G \big\> = \big\< \sigma_k \cdot \nabla\omega, D G\big\>.$$
\end{lemma}

\begin{proof}
Assume $G= g\circ \Pi_\Lambda$ for some finite set $\Lambda \Subset \Z^2$. Note that
  $$\big\< \sigma_k \cdot \nabla\omega, D G \big\> = - \sum_{j\in \Lambda} (\partial_j g) (\Pi_\Lambda\omega) \< \omega, \sigma_k \cdot \nabla e_j\> = - \sum_{j\in \Lambda} (\partial_j g) (\Pi_\Lambda\omega) \< \Pi_N \omega, \sigma_k \cdot \nabla e_j\>$$
for all large $N$. Hence
  \[ \big\< \sigma_k \cdot \nabla\omega, DG \big\> = \sum_{j\in \Lambda} (\partial_j g) (\Pi_\Lambda\omega) \< \sigma_k \cdot \nabla (\Pi_N \omega), e_j\> = \big\<\sigma_k \cdot \nabla (\Pi_N \omega), D G\big\>. \qedhere \]
\end{proof}

\subsection{The case $\gamma>2$}

Our purpose is to prove the assertion (iii) of Theorem \ref{main-result}. Before proceeding further, we need some technical preparations. For $k,l\in \Z^2_0$, set
  \begin{equation}\label{constant}
  C_{k,l}(\gamma)= \frac{k^\perp \cdot l}{|k|^\gamma}.
  \end{equation}
We shall omit the parameter $\gamma$ in $C_{k,l}(\gamma)$ to save space.

\begin{lemma}\label{lem-2-1}
For any $l\in \Z_0^2$,
  $$\nabla e_l = 2\pi l \, e_{-l}.$$
As a result, for any $k,l\in \Z^2_0$,
  $$\sigma_k \cdot \nabla e_l = \sqrt{2} \pi C_{k,l} \, e_k e_{-l}.$$
Moreover,
  \begin{equation}\label{lem-2-1.1}
  \sum_{|k|\leq N} C_{k,l}^2 = \frac12 a_N |l|^2 \quad \mbox{with} \quad a_N = \sum_{|k|\leq N} \frac1{|k|^{2(\gamma -1)}}.
  \end{equation}
\end{lemma}

\begin{proof}
If $l\in \Z^2_+$, then
  $$\aligned
  \nabla e_l(x) &= \sqrt{2}\, \nabla \cos(2\pi l\cdot x)= -\sqrt{2} \sin(2\pi l\cdot x)\, 2\pi l = 2\pi l \sqrt{2} \sin(2\pi (-l)\cdot x) = 2\pi l \, e_{-l}(x).
  \endaligned $$
If $k\in \Z^2_-$, then
  $$\aligned
  \nabla e_l(x) &= \sqrt{2}\, \nabla \sin(2\pi l\cdot x)= \sqrt{2} \cos(2\pi l\cdot x)\, 2\pi l = 2\pi l \sqrt{2} \cos(2\pi (-l)\cdot x) = 2\pi l \, e_{-l}(x).
  \endaligned$$
Therefore,
  $$\sigma_k(x) \cdot \nabla e_l(x)= \bigg(\frac1{\sqrt{2}}\frac{k^\perp}{|k|^\gamma} e_k(x) \bigg) \cdot \big(2\pi l \, e_{-l}(x)\big) = \sqrt{2} \pi C_{k,l} \, e_k(x) e_{-l}(x). $$

It remains to prove the last result. Denoting by $D_{k,l}= \frac{k\cdot l}{|k|^\gamma}$, then
  $$C_{k,l}^{2} + D_{k,l}^{2} = \frac{(k^\perp \cdot l)^2}{|k|^{2\gamma}} + \frac{(k \cdot l)^2}{|k|^{2\gamma}}= \frac1{|k|^{2(\gamma -1)}} \bigg[ \bigg( \frac{k^\perp}{|k|} \cdot l\bigg)^2 + \bigg( \frac{k}{|k|} \cdot l\bigg)^2 \bigg]= \frac{|l|^2}{|k|^{2(\gamma -1)}}.$$
The transformation $k\to k^\perp$ is 1-1 on the set $\{k\in \Z^2_0: |k|\leq N\}$, and preserves the norm $|\cdot|$. As a result,
  $$\sum_{|k|\leq N} C_{k,l}^{2} = \sum_{|k|\leq N} \frac{(k^\perp \cdot l)^2}{|k|^{2\gamma}} = \sum_{|k|\leq N} \frac{((k^\perp)^\perp \cdot l)^2}{|k^\perp|^{2\gamma}}= \sum_{|k|\leq N} \frac{(k \cdot l)^2}{|k|^{2\gamma}} = \sum_{|k|\leq N} D_{k,l}^{2}.$$
Combining the above two equalities, we obtain
  \[  \sum_{|k|\leq N} C_{k,l}^{2} = \frac12 \sum_{|k|\leq N} \big( C_{k,l}^{2} + D_{k,l}^{2} \big)= \frac12 |l|^2 \sum_{|k|\leq N} \frac1{|k|^{2(\gamma -1)}} = \frac12 a_N |l|^2. \qedhere \]
\end{proof}

Now we can prove the next result which characterizes the integrability of directional derivatives for cylindrical functionals.

\begin{proposition}\label{prop-gradient-cylinder}
For any cylindrical function $F$,
  $$ \sum_{k\in \Z^2_0} \int \< \sigma_k\cdot \nabla \omega, D F \>^2 \,\d\mu <+\infty$$
if and only if $\gamma>2$ in \eqref{vector-fields}.
\end{proposition}

\begin{proof}
Assume that $F$ is of the form $F = f \circ \Pi_\Lambda$ for some $\Lambda \Subset \Z^2$ and $f\in C^\infty(H_\Lambda)$. By Lemma \ref{lem-2-1},
  $$\< \sigma_k\cdot \nabla \omega, D F \> = - \sum_{l\in \Lambda} (\partial_l f)(\Pi_\Lambda \omega) \<\omega, \sigma_k\cdot \nabla e_l\> = -\sqrt{2} \pi \sum_{l\in \Lambda} C_{k,l}\, (\partial_l f)(\Pi_\Lambda \omega) \<\omega, e_k e_{-l}\>, $$
which implies
  \begin{equation}\label{prop-gradient-cylinder.1}
  \< \sigma_k\cdot \nabla \omega, D F \>^2 = 2\pi^2 \sum_{l, m\in \Lambda} C_{k,l} C_{k,m}\, (\partial_l f)(\Pi_\Lambda \omega) (\partial_m f) (\Pi_\Lambda \omega)\, \<\omega, e_k e_{-l}\> \<\omega, e_k e_{-m}\>.
  \end{equation}

Assume $k,l \in \Z^2_+$; we have $-l\in \Z^2_-$ and
  $$e_k(x) e_{-l}(x) = - 2\cos(2\pi k\cdot x) \sin(2\pi l\cdot x) = \sin(2\pi (k-l)\cdot x) - \sin(2\pi (k+l)\cdot x).  $$
Note that $k+l\in \Z^2_+$, thus $- \sin(2\pi (k+l)\cdot x) = \frac1{\sqrt 2} e_{-k-l}(x)$. Next,
  $$\sin(2\pi (k-l)\cdot x) = \frac1{\sqrt 2} \begin{cases}
  -e_{-k+l}(x), & \mbox{if } k-l\in \Z^2_+ ; \\
  e_{k-l}(x), & \mbox{if } k-l\in \Z^2_-.
  \end{cases} $$
Similarly, we can get the expression of $e_k e_{-l}$ for $k\in \Z^2_-$ or $l \in \Z^2_-$. Since $\Lambda$ is fixed and $l\in \Lambda$, we see that $\pm k \pm l \notin \Lambda$ when $|k|$ is big enough, which implies that $(\partial_m f)(\Pi_\Lambda \omega)$ and $\<\omega, e_k e_{-l}\>$ are independent as r.v.s on $H^{-1-}$. Hence, by \eqref{prop-gradient-cylinder.1}, for all $k\in \Z^2_0$ with $|k|$ big enough,
  $$ \aligned
  \int \< \sigma_k\cdot \nabla \omega, D F \>^2\,\d\mu &= 2\pi^2 \sum_{l, m\in \Lambda} C_{k,l} C_{k,m} \int (\partial_l f)(\Pi_\Lambda \omega) (\partial_m f )(\Pi_\Lambda \omega) \,\d\mu \\
  &\hskip13pt \times \int \<\omega, e_k e_{-l}\> \<\omega, e_k e_{-m}\> \,\d\mu.
  \endaligned $$
  Note that
  $$\int \<\omega, e_k e_{-l}\> \<\omega, e_k e_{-m}\> \,\d\mu = \int_{\T^2} e_k^2 e_{-l}e_{-m} \,\d x,$$
thus, for any $M, N\in \N,\, M>N \gg 1$,
  \begin{equation}\label{prop-gradient-cylinder.2}
  \aligned
  \sum_{N< |k|\leq M}\int \< \sigma_k\cdot \nabla \omega, D F \>^2\,\d\mu &= 2\pi^2 \sum_{l, m\in \Lambda} \int (\partial_l f)(\Pi_\Lambda \omega) (\partial_m f )(\Pi_\Lambda \omega) \,\d\mu \\
  &\hskip13pt \times \sum_{N< |k|\leq M}  C_{k,l} C_{k,m} \int_{\T^2} e_k^2 e_{-l}e_{-m} \,\d x.
  \endaligned
  \end{equation}

Since $C_{-k,l} = -C_{k,l}$ and $e_k^2 + e_{-k}^2 \equiv 2$, we have
  $$\aligned \sum_{N< |k|\leq M}  C_{k,l} C_{k,m} e_k^2 &= \sum_{N< |k|\leq M, k\in \Z^2_+}  \big(C_{k,l} C_{k,m} e_k^2 + C_{-k,l} C_{-k,m} e_{-k}^2\big) \\
  &=  \sum_{N< |k|\leq M, k\in \Z^2_+} 2C_{k,l} C_{k,m} = \sum_{N< |k|\leq M} C_{k,l} C_{k,m}.
  \endaligned$$
Therefore, by \eqref{lem-2-1.1},
  \begin{equation}\label{prop-gradient-cyl.3}
  \sum_{N< |k|\leq M}  C_{k,l} C_{k,m} \int_{\T^2} e_k^2 e_{-l}e_{-m} \,\d x = \delta_{l,m} \sum_{N< |k|\leq M}  C_{k,l}^2 = \frac12 \delta_{l,m} (a_M- a_N) |l|^2.
  \end{equation}
Substituting this result into \eqref{prop-gradient-cylinder.2} leads to
  $$\sum_{N< |k|< M}\int \< \sigma_k\cdot \nabla \omega, D F \>^2\,\d\mu = \pi^2 (a_M- a_N) \sum_{l\in \Lambda} |l|^2 \int \big[ (\partial_l f)(\Pi_\Lambda \omega) \big]^2 \,\d\mu .$$
The definition of $a_N$ immediately implies the desired assertion.
\end{proof}

Finally we can prove in the case $\gamma >2$ the existence of equations to the forward Kolmogorov equation \eqref{Kolmog-eq}. The proof relies on the convergence result of the drift part proved in Section \ref{appendix-drift}, see Theorem \ref{thm-app}.

\begin{proof}[Proof of Theorem \ref{main-result}(iii)]
Fix a test functional $F \in \mathcal{FC}_{P}$ which can also be considered as a function on $H_N$ for all $N$ such that $\Lambda_N \supset \Lambda$. Let $\alpha\in C^1([0,T])$ with $\alpha(T)=0$. Multiplying both sides of \eqref{Kolmogorov-eq} by $\alpha F$ and integrating by parts on $[0,T]\times H_N$ yields
  $$\aligned
  0= &\ \alpha(0)\int_{H_N} F \rho^N_0 \,\d\mu_N + \int_0^T \!\! \int_{H_N} \rho^N_s \big( \alpha'(s) F - \alpha(s)\<b_N, \nabla_N F\> \big) \,\d\mu_N\d s  \\
  & - \frac12 \sum_{|k|\leq N} \int_0^T \!\! \int_{H_N} \alpha(s) \big\<G_N^k, \nabla_N F\big\>_{H_N} \big\<G_N^k, \nabla_N \rho^N_s\big\>_{H_N} \,\d\mu_N\d s.
  \endaligned$$
Changing $N$ to $N_i$ obtained in the proof of Theorem \ref{thm-1}, this is equivalent to
  \begin{equation}\label{thm-existence-Kol-eq.1}
  \aligned
  0= &\ \alpha(0)\int F \rho^{N_i}_0 \,\d\mu + \int_0^T \!\! \int \rho^{N_i}_s \big( \alpha'(s) F - \alpha(s) \big\< u(\Pi_{N_i} \omega) \cdot \nabla(\Pi_{N_i} \omega), D F \big\> \big) \,\d\mu\d s  \\
  &- \frac12 \sum_{|k|\leq N_i} \int_0^T \!\! \int \alpha(s) \big\< \sigma_k \cdot \nabla(\Pi_{N_i} \omega), D F\big\> \big\< \sigma_k \cdot \nabla(\Pi_{N_i} \omega), D \rho^{N_i}_s \big\> \,\d\mu\d s.
  \endaligned
  \end{equation}

We want to show that all the terms on the r.h.s. converge to the corresponding ones.  The assertion (a) in the proof of Theorem \ref{thm-1} enables us to let $i \to \infty$ in the first two terms. For the third one involving the drift part, we have
  \begin{equation*}
  \aligned
  &\hskip13pt \int_0^T \!\! \int \rho^{N_i}_s \alpha(s) \big\< u(\Pi_{N_i} \omega) \cdot \nabla(\Pi_{N_i} \omega), D F \big\> \,\d\mu\d s - \int_0^T \!\! \int \rho_s\, \alpha(s) \big\< u(\omega) \cdot \nabla\omega, D F \big\>\,\d\mu\d s \\
  & = \int_0^T \!\! \int \rho^{N_i}_s\, \alpha(s) \Big(\big\< u(\Pi_{N_i} \omega) \cdot \nabla(\Pi_{N_i} \omega), D F \big\> - \big\< u(\omega) \cdot \nabla\omega, D F \big\> \Big)\,\d\mu\d s \\
  &\hskip13pt + \int_0^T \!\! \int \big(\rho^{N_i}_s - \rho_s \big) \alpha(s) \big\< u(\omega) \cdot \nabla\omega, D F \big\>  \,\d\mu\d s \\
  & =: I^{N_i}_1 + I^{N_i}_2.
  \endaligned
  \end{equation*}
Under our assumption on $F$, we deduce from Theorem \ref{thm-app} that $I^{N_i}_1$ tends to 0 as $i\to \infty$, since $\big\{ \rho^{N_i} \big\}_{i\in \N}$ is bounded in $ L^\infty \big([0,T], L^2(H^{-1-}, \mu) \big)$. Still by Theorem \ref{thm-app}, we have $\big\< u(\omega) \cdot \nabla\omega, D F \big\> \in L^2(H^{-1-}, \mu)$, hence the second term also converges to 0 as  $i\to \infty$, thanks to (a) in the proof of Theorem \ref{thm-1}. Therefore, we obtain the convergence of the third integral in \eqref{thm-existence-Kol-eq.1}.

Finally, we deal with the last integral in \eqref{thm-existence-Kol-eq.1}. Fix any $n\in \N$, we denote by
  $$\aligned
  J^{N_i}_1 = & \sum_{k\in \Lambda_n}  \int_0^T \!\! \int \alpha(s) \big\< \sigma_k \cdot \nabla(\Pi_{N_i} \omega), D F \big\> \big\< \sigma_k \cdot \nabla(\Pi_{N_i} \omega), D \rho^{N_i}_s \big\> \,\d\mu\d s\\
  & -  \sum_{k\in \Lambda_n} \int_0^T \!\! \int \alpha(s) \< \sigma_k \cdot \nabla \omega, D F\> \< \sigma_k \cdot \nabla \omega, D \rho_s \> \,\d\mu\d s
  \endaligned$$
and
  $$\aligned
  J^{N_i}_2 =& \sum_{k\in \Lambda_{N_i} \setminus \Lambda_n}  \int_0^T \!\! \int \alpha(s) \big\< \sigma_k \cdot \nabla(\Pi_{N_i} \omega), D F \big\> \big\< \sigma_k \cdot \nabla(\Pi_{N_i} \omega), D \rho^{N_i}_s \big\> \,\d\mu\d s \\
  & - \sum_{k\in \Lambda_n^c} \int_0^T \!\! \int \alpha(s) \< \sigma_k \cdot \nabla \omega, D F\> \< \sigma_k \cdot \nabla \omega, D \rho_s \> \,\d\mu\d s ,
  \endaligned$$
where $\Lambda_n^c= \Z^2 \setminus \Lambda_n$. Since $n$ is fixed and $k\in \Lambda_n$, Lemma  \ref{lem-3} implies that, for all $i$ big enough,
  $$\big\< \sigma_k \cdot \nabla(\Pi_{N_i} \omega), D F\big\>  = \< \sigma_k \cdot \nabla \omega, D F\> \in \mathcal{FC}_P.$$
Therefore,
  $$J^{N_i}_1 = \sum_{k\in \Lambda_n}  \int_0^T \!\! \int \alpha(s) \big\< \sigma_k \cdot \nabla \omega, D F\big\> \Big(\big\< \sigma_k \cdot \nabla(\Pi_{N_i} \omega), D \rho^{N_i}_s \big\> - \< \sigma_k \cdot \nabla \omega, D \rho_s \> \Big) \,\d\mu\d s,$$
which tends to 0 by (b) in the proof of Theorem \ref{thm-1} and \eqref{thm-1.2}. Regarding the term $J^{N_i}_2$, by Cauchy's inequality, \eqref{projection} and the estimates \eqref{estimate-3}, \eqref{thm-1.0}, we have
  $$\aligned
  \big|J^{N_i}_2 \big| &\leq \big(1+ \|\rho_0 \|_{L^2(\mu)} \big) T\|\alpha\|_\infty \bigg[\sum_{k\in \Lambda_{N_i} \setminus \Lambda_n} \int \big\< \sigma_k \cdot \nabla(\Pi_{N_i} \omega), D F\big\> ^2 \,\d\mu \bigg]^{1/2} \\
  & \hskip13pt + \|\rho_0 \|_{L^2(\mu)} T\|\alpha\|_\infty \bigg[\sum_{k\in \Lambda_n^c} \int \< \sigma_k \cdot \nabla \omega, D F \>^2 \,\d\mu \bigg]^{1/2}.
  \endaligned   $$
By the equality at the end of the proof of Proposition \ref{prop-gradient-cylinder},
  $$\sum_{k\in \Lambda_n^c} \int \< \sigma_k \cdot \nabla \omega, D F \>^2 \,\d\mu \leq C \sum_{k\in \Lambda_n^c} \frac{1}{|k|^{2\gamma -2}}.$$
Following the proof of Proposition \ref{prop-gradient-cylinder}, one can show that
  $$\sup_{i\geq 1} \sum_{k\in \Lambda_{N_i} \setminus \Lambda_n} \int \big\< \sigma_k \cdot \nabla(\Pi_{N_i} \omega), D F \big\>^2 \,\d\mu \leq C \sum_{k\in \Lambda_n^c} \frac{1}{|k|^{2\gamma -2}}.$$
Therefore,
  $$ \sup_{i\geq 1}  \big|J^{N_i}_2 \big| \leq 2 \big(1+ \|\rho_0 \|_{L^2(\mu)} \big) T\|\alpha\|_\infty C \sum_{k\in \Lambda_n^c} \frac{1}{|k|^{2\gamma -2}} $$
which vanishes as $n\to \infty$. Summarizing these arguments we conclude that the last term in \eqref{thm-existence-Kol-eq.1} also converges to the corresponding quantity.
\end{proof}

\section{The case $\gamma=2$} \label{proof-second-main-result}

Now we turn to prove Theorem \ref{main-result-2} for which we need some preparations. We find that there is a small technical problem which prevents us from applying directly the approximation arguments in Section \ref{sec-finite-Kol-eq} and those results in Theorem \ref{thm-1}. The reason is that we are unable to prove a decomposition formula, similar to that in Proposition \ref{prop-operator}, for the diffusion part of \eqref{approx-generator}. Therefore, we slightly modify the approximating operator $\L_N$ defined in \eqref{approx-generator} as follows:
  \begin{equation}\label{approx-generator-1}
  \tilde \L_N \phi(\xi)= \frac12 \sum_{k\in \Gamma_{N}} \Big\<G_N^k, \nabla_N \big\<G_N^k, \nabla_N \phi\big\>_{H_N} \Big\>_{\! H_N}(\xi) - \<b_N, \nabla_N \phi\>_{H_N}(\xi),
  \end{equation}
where
  $$\Gamma_{N} = \big\{ k\in \Z^2_0: |k| \leq N/3 \big\}.$$
The idea for this modification will be clear in view of the beginning of the proof of Theorem \ref{main-result-2} given below. Indeed, we can replace $1/3$ in the definition of $\Gamma_{N}$ by any constant $\theta\in(0, 1/2)$.

Assume the condition \eqref{projection}. We consider the new finite dimensional Kolmogorov equations
  \begin{equation}\label{Kolmogorov-eq-new}
  \partial_t \rho^N_t = \tilde \L_N^\ast \rho^N_t, \quad \rho^N|_{t=0}= \rho^N_0\in C_P^\infty(H_N),
  \end{equation}
and repeat the discussions in Sections \ref{sec-finite-Kol-eq} and \ref{proof-first-main-result}. All the arguments hold true in this case with little change, and we obtain a measurable function $\rho\in L^\infty \big( [0,T], L^2( H^{-1-}) \big)$, satisfying the properties (i) and (ii) in Theorem \ref{main-result}. To show that $\rho$ is in fact a trivial function, we need more preparations.

Recall the definition \eqref{real-basis} of the real basis of $L^2(\T^2)$. We have $\sigma_k(x) = \frac1{\sqrt{2}} \frac{k^\perp}{|k|^2} e_k(x),\, k\in \Z^2_0$. Recall also the definition of $C_{k,l}$ in \eqref{constant} and keep in mind that $\gamma =2$. For a cylindrical function $F(\omega)= f(\Pi_\Lambda \omega)$ with some $\Lambda \Subset \Z_0^2$, we define
  $$\L_N^{0} F(\omega) = \frac12 \sum_{k\in \Gamma_{N}} \big\< \sigma_k\cdot \nabla\omega, D \< \sigma_k\cdot \nabla\omega, D F\> \big\>.$$

\begin{lemma}\label{lem-2-2}
We have
  \begin{equation}\label{lem-2-2.1}
  \aligned
  \L_N^{0} F(\omega) &= \pi^2 \sum_{k\in \Gamma_{N}} \sum_{l, m\in \Lambda} C_{k,l} C_{k,m}\, f_{l,m}(\omega) \<\omega, e_k e_{-l}\> \<\omega, e_k e_{-m}\> \\
  &\hskip11pt -\pi^2 \sum_{k\in \Gamma_{N}} \sum_{l\in \Lambda} C_{k,l}^2 \, f_l(\omega) \big\< \omega, e_k^2 e_l \big\> ,
  \endaligned
  \end{equation}
where we write $f_l(\omega)= (\partial_l f)(\Pi_\Lambda \omega)$ and $f_{l,m}(\omega)= (\partial_{l,m} f)(\Pi_\Lambda \omega)$ to simplify notations.
\end{lemma}

\begin{proof}
Note that $DF(\omega )= \sum_{l\in \Lambda} (\partial_l f)(\Pi_\Lambda \omega) e_l = \sum_{l\in \Lambda} f_l(\omega) e_l$; therefore, by Lemma \ref{lem-2-1},
  $$\aligned
  \<\sigma_k \cdot \nabla\omega, D F\> &= \sum_{l\in \Lambda} f_l(\omega) \<\sigma_k \cdot \nabla\omega, e_l \> = - \sum_{l\in \Lambda} f_l(\omega) \<\omega, \sigma_k \cdot \nabla e_l \> \\
  &= -\sqrt{2} \pi \sum_{l\in \Lambda} C_{k,l}\, f_l(\omega) \<\omega, e_k e_{-l}\>.
  \endaligned$$
Furthermore,
  $$\aligned
  D \<\sigma_k \cdot \nabla\omega, D F\> &= -\sqrt{2}\pi \sum_{l\in \Lambda} C_{k,l} \big( \<\omega, e_k e_{-l}\> D [ f_l(\omega) ] + f_l(\omega) e_k e_{-l} \big) \\
  &= -\sqrt{2}\pi \sum_{l, m\in \Lambda} C_{k,l} \<\omega, e_k e_{-l}\> f_{l,m}(\omega) e_m  - \sqrt{2}\pi \sum_{l\in \Lambda} C_{k,l}\, f_l(\omega) e_k e_{-l}.
  \endaligned$$
As a result,
  \begin{equation}\label{lem-2-2.2}
  \aligned
  \big\< \sigma_k\cdot \nabla\omega, D \< \sigma_k\cdot \nabla\omega, D F\> \big\> & = -\sqrt{2}\pi \sum_{l, m\in \Lambda} C_{k,l}\, f_{l,m}(\omega) \<\omega, e_k e_{-l}\> \< \sigma_k\cdot \nabla\omega, e_m \> \\
  &\hskip11pt - \sqrt{2}\pi \sum_{l\in \Lambda} C_{k,l}\, f_l(\omega) \< \sigma_k\cdot \nabla\omega, e_k e_{-l} \>.
  \endaligned
  \end{equation}
By Lemma \ref{lem-2-1}, we have $\< \sigma_k\cdot \nabla\omega, e_m \> = - \sqrt{2}\pi C_{k,m} \<\omega, e_k e_{-m}\>$ and
  $$\< \sigma_k\cdot \nabla\omega, e_k e_{-l} \> = - \<\omega, \sigma_k\cdot \nabla(e_k e_{-l}) \> = - \sqrt{2}\pi C_{k,-l} \big\<\omega, e_k^2 e_l \big\> = \sqrt{2}\pi C_{k,l} \big\<\omega, e_k^2 e_l \big\>.$$
Substituting these facts into \eqref{lem-2-2.2} and summing over $k$ yield the desired result.
\end{proof}

We shall rewrite $\L_N^{0} F(\omega)$ as the sum of two parts, in which one part is convergent while the other is in general divergent.

\begin{proposition}\label{prop-operator}
It holds that
  \begin{equation}\label{prop-operator.1}
  \aligned
  \L_N^{0} F(\omega) &= \pi^2 \sum_{l, m\in \Lambda} f_{l,m}(\omega) \sum_{k\in \Gamma_{N}} C_{k,l} C_{k,m} \big( \<\omega, e_k e_{-l}\> \<\omega, e_k e_{-m}\> - \delta_{l,m} \big) \\
  &\hskip11pt +\frac12 \pi^2 b_N \sum_{l\in \Lambda} |l|^2 \big[f_{l,l}(\omega) - f_l(\omega) \< \omega, e_l \>\big],
  \endaligned
  \end{equation}
where
  $$b_N = \sum_{k\in \Gamma_{N}} \frac1{|k|^2} .$$
Moreover, for any $l,m\in \Z_0^2$, the quantity
  \begin{equation}\label{prop-operator.2}
  R_{l,m}(N) =\sum_{k\in \Gamma_{N}} C_{k,l} C_{k,m} \big( \<\omega, e_k e_{-l}\> \<\omega, e_k e_{-m}\> - \delta_{l,m} \big)
  \end{equation}
is a Cauchy sequence in $L^p(H^{-1-}, \mu)$ for any $p>1$.
\end{proposition}

\begin{proof}
Here we only prove the equality \eqref{prop-operator.1}. The proof of the second assertion involves lots of detailed computations and is postponed to the appendix.

We have
  $$\aligned
  &\hskip12pt \sum_{k\in \Gamma_{N}} \sum_{l, m\in \Lambda} C_{k,l} C_{k,m}\, f_{l,m}(\omega) \<\omega, e_k e_{-l}\> \<\omega, e_k e_{-m}\> \\
  &= \sum_{l,m\in \Lambda} f_{l,m}(\omega) \sum_{k\in \Gamma_{N}} C_{k,l} C_{k,m} \big( \<\omega, e_k e_{-l}\> \<\omega, e_k e_{-m}\> -\delta_{l,m} \big) + \sum_{l\in \Lambda} f_{l,l}(\omega) \sum_{k\in \Gamma_{N}} C_{k,l}^2.
  \endaligned$$
Analogous to \eqref{lem-2-1.1}, we have
  $$\sum_{k\in \Gamma_{N}} C_{k,l}^2 =\frac12 b_N |l|^2.$$
Therefore,
  \begin{equation}\label{cor-3}
  \aligned
  &\hskip12pt \sum_{k\in \Gamma_{N}} \sum_{l,m\in \Lambda} C_{k,l} C_{k,m} f_{l,m}(\omega) \< \omega, e_{k} e_{-l} \> \< \omega, e_{k} e_{-m} \> \\
  &= \sum_{l,m\in \Lambda} f_{l,m}(\omega) \sum_{k\in \Gamma_{N}} C_{k,l} C_{k,m} \big( \< \omega, e_{k} e_{-l} \> \< \omega, e_{k} e_{-m} \> - \delta_{l,m} \big) + \frac12 b_N \sum_{l\in \Lambda} |l|^2 f_{l,l}(\omega) .
  \endaligned
  \end{equation}

Next, note that $C_{-k,l} = -C_{k,l}$ and $e_{k}^2 + e_{-k}^2 \equiv 2$ for all $k\in \Z^2_0$, we have
  \begin{equation}\label{cor-4}
  \aligned \sum_{k\in \Gamma_{N}} C_{k,l}^{2} \big\langle \omega, e_{k}^2 e_{l} \big\rangle &=  \sum_{k\in \Gamma_{N}, k\in \Z^2_+} \big[ C_{k,l}^{2} \big\langle \omega, e_{k}^2 e_{l} \big\rangle + C_{-k,l}^{2} \big\langle \omega, e_{-k}^2 e_{l} \big\rangle \big] \\
  &= \sum_{k\in \Gamma_{N}, k\in \Z^2_+} 2 C_{k,l}^{2} \<\omega, e_l\> = \frac12 b_N |l|^2 \<\omega, e_l\> .
  \endaligned
  \end{equation}
Hence,
  $$\sum_{k\in \Gamma_{N}} \sum_{l\in\Lambda} C_{k,l}^{2}f_{l}(\omega) \big\< \omega, e_{k}^2 e_{l} \big\> = \frac12 b_N \sum_{l\in\Lambda} |l|^2 f_{l}(\omega) \<\omega, e_l\>. $$
Combining this equality with \eqref{lem-2-2.1} and \eqref{cor-3} leads to the desired identity \eqref{prop-operator.1}.
\end{proof}

Applying Proposition \ref{prop-operator} to Hermite polynomials yields

\begin{proposition}\label{prop-operator-hermite}
For any fixed $\bm{n}\in \bm{N}$, we have the decomposition
  $$\L_N^{0} H_{\bm{n}}(\omega) = I_N(\omega) - C_{\bm{n}} b_N H_{\bm{n}}(\omega), $$
where the sequence $\{I_N\}_{N\geq 1}$ is convergent in $L^2(H^{-1-}, \mu)$ and $C_{\bm{n}}= \frac12 \pi^2 \sum_{l\in \Z_0^2} n_l |l|^2 <\infty$.
\end{proposition}

\begin{proof}
Let $\Lambda=\{l\in \Z_0^2: n_l\geq 1\}$ and $f(x)= \prod_{l\in \Lambda} h_{n_l}(x_l),\, x\in \R^\Lambda$. Then $H_{\bm{n}}(\omega) = f\circ \Pi_\Lambda (\omega)$. In view of the decomposition in \eqref{prop-operator.1}, it is natural to set
  $$I_N(\omega) =\pi^2 \sum_{l, m\in \Lambda} f_{l,m}(\omega) R_{l,m}(N).$$
Since $f(x)$ is a polynomial on $\R^\Lambda$, it is clear that $f_{l,m}(\omega) = (\partial_{l,m} f)(\Pi_\Lambda \omega)$ is integrable of any order $p>1$. Combining this with the last assertion of Proposition \ref{prop-operator} yields the convergence property of $I_N$.

To obtain the second part of the decomposition, we note that, for any $l\in \Lambda$,
  $$\aligned
  \frac{\partial^2}{\partial x_l^2}f - x_l \frac{\partial}{\partial x_l}f &= \bigg( \prod_{j\in \Lambda\setminus \{l\}} h_{n_j}(x_j) \bigg) \big(h''_{n_l}(x_l) - x_l h'_{n_l}(x_l)\big) \\
  &= \bigg( \prod_{j\in \Lambda\setminus \{l\}} h_{n_j}(x_j) \bigg) \big( -n_l h_{n_l}(x_l) \big) = -n_l f(x).
  \endaligned$$
This immediately gives us the desired result.
\end{proof}

Finally we can present

\begin{proof}[Proof of Theorem \ref{main-result-2}]
As mentioned at the beginning of this subsection, we can construct a  sequence of functions $\rho^{N}$, which solve \eqref{Kolmogorov-eq-new} and contain a subsequence $\rho^{N_i}$ converging weakly-$\ast$ to some limit $\rho\in L^\infty \big([0,T], L^2(H^{-1-}, \mu)\big)$. We shall prove the projections of $\rho$ on nontrivial Hermite polynomials vanish.

Let $F= H_{\bm{n}}$ for some $\bm{n}\in \bm{N}$ and $\alpha\in C^1([0,T],\R)$ with $\alpha(T)=0$. Similar to the proof of Theorem \ref{main-result}, we still have \eqref{thm-existence-Kol-eq.1}, with the only difference of summing over $k\in \Gamma_{N_i}$. We integrate by parts the last integral in \eqref{thm-existence-Kol-eq.1} and obtain
  \begin{equation}\label{proof.1}
  \aligned
  0= &\ \alpha(0)\int H_{\bm{n}} \rho^{N_i}_0 \,\d\mu + \int_0^T \!\! \int \rho^{N_i}_s \big( \alpha'(s) H_{\bm{n}} - \alpha(s) \big\< u(\Pi_{N_i} \omega) \cdot \nabla(\Pi_{N_i} \omega), D H_{\bm{n}} \big\> \big) \,\d\mu\d s  \\
  & + \frac12 \sum_{k\in \Gamma_{N_i}} \int_0^T \!\! \int \alpha(s) \rho^{N_i}_s\, \Big\< \sigma_k \cdot \nabla(\Pi_{N_i} \omega), D \big\< \sigma_k \cdot \nabla(\Pi_{N_i} \omega), D H_{\bm{n}}\big\> \Big\> \,\d\mu\d s.
  \endaligned
  \end{equation}
For the given $\bm{n}\in \bm{N}$, let $\Lambda = \{l\in \Z^2_0: n_l\geq 1\} $. In this case, we say that $H_{\bm{n}}$ is $H_\Lambda$-measurable. Of course, $H_{\bm{n}}$ is also $H_{\Lambda'}$-measurable for any $\Lambda' \supset \Lambda$. When $i$ is big enough, we have $\Lambda \subset \Gamma_{N_i} = \big\{ k\in \Z_0^2: |k| \leq N_i/3 \big\}$. Then, similar to Lemma \ref{lem-3}, for all $k\in \Gamma_{N_i}$,
  $$\big\< \sigma_k \cdot \nabla(\Pi_{N_i} \omega), D H_{\bm{n}}\big\> = - \big\< \Pi_{N_i} \omega, \sigma_k \cdot \nabla (D H_{\bm{n}}) \big\> = - \big\< \omega, \sigma_k \cdot \nabla (D H_{\bm{n}}) \big\> = \< \sigma_k \cdot \nabla \omega, D H_{\bm{n}} \>.$$
We see that $\< \sigma_k \cdot \nabla \omega, D H_{\bm{n}} \>$ is $H_{\Lambda_{2N_i/3}}$-measurable. In the same way, we have
  $$\aligned
  \Big\< \sigma_k \cdot \nabla(\Pi_{N_i} \omega), D \big\< \sigma_k \cdot \nabla(\Pi_{N_i} \omega), D H_{\bm{n}}\big\> \Big\> &= \Big\< \sigma_k \cdot \nabla(\Pi_{N_i} \omega), D \< \sigma_k \cdot \nabla \omega, D H_{\bm{n}}\> \Big\> \\
  &= \big\< \sigma_k \cdot \nabla \omega, D\< \sigma_k \cdot \nabla\omega, D H_{\bm{n}}\> \big\>,
  \endaligned$$
which is $H_{\Lambda_{N_i}}$-measurable. Therefore,
  $$\frac12 \sum_{k\in \Gamma_{N_i}} \Big\< \sigma_k \cdot \nabla(\Pi_{N_i} \omega), D \big\< \sigma_k \cdot \nabla(\Pi_{N_i} \omega), D H_{\bm{n}}\big\> \Big\> = \mathcal{L}_{N_i}^0 H_{\bm{n}}(\omega).$$
By Proposition \ref{prop-operator-hermite}, we can rewrite \eqref{proof.1} as
  \begin{equation}\label{proof.3}
  \aligned
  0= &\ \alpha(0)\int H_{\bm{n}} \rho^{N_i}_0 \,\d\mu + \int_0^T \!\! \int \rho^{N_i}_s \big( \alpha'(s) H_{\bm{n}} - \alpha(s) \big\< u(\Pi_{N_i} \omega) \cdot \nabla(\Pi_{N_i} \omega), D H_{\bm{n}} \big\> \big) \,\d\mu\d s  \\
  & + \int_0^T \!\! \int \alpha(s) \rho^{N_i}_s I_{N_i} \,\d\mu\d s - C_{\bm{n}} b_{N_i} \int_0^T \!\! \int \alpha(s) \rho^{N_i}_s H_{\bm{n}} \,\d\mu\d s.
  \endaligned
  \end{equation}

Repeating the arguments in the proof Theorem \ref{main-result}(iii), we see that all the integrals in the first line of \eqref{proof.3} are convergent to the corresponding terms. Moreover, since $I_{N_i}$ converges in $L^2(H^{-1-}, \mu)$, the first integral in the second line is also convergent. One the other hand, the weak-$\ast$ convergence of $\rho^{N_i}$ implies
  $$\lim_{i\to \infty} \int_0^T \!\! \int \alpha(s) \rho^{N_i}_s H_{\bm{n}} \,\d\mu\d s = \int_0^T \!\! \int \alpha(s) \rho_s H_{\bm{n}} \,\d\mu\d s.$$
Since $b_{N_i}$ tends to infinity, the limit above must be 0. By the arbitrariness of $\alpha$ we deduce that, for any $\bm{n}\in \bm{N}$ with $\bm{n} \neq 0$, it holds
  $$ \int \rho_s H_{\bm{n}} \,\d\mu =0 \quad \mbox{for a.e. } s\in (0,T). $$
This shows that $\rho_s$ is constant for a.e. $s\in (0,T)$.

In order to show that the constant is the same for different $s$, replacing $H_{\bm{n}}$ in \eqref{proof.1} by $F\equiv 1$ leads to
  $$0= \alpha(0)\int \rho^{N_i}_0 \,\d\mu + \int_0^T \!\! \int \rho^{N_i}_s \alpha'(s) \,\d\mu\d s.$$
Letting $i\to \infty$ and using the weak-$\ast$ convergence of $\rho^{N_i}$, we get
  $$0= \alpha(0)\int \rho_0 \,\d\mu + \int_0^T \!\! \int \rho_s\, \alpha'(s) \,\d\mu\d s .$$
This implies that $\frac{\d}{\d s} \int \rho_s\,\d\mu =\frac{\d}{\d s} \rho_s =0$ on $(0,T)$ in the distributional sense. The proof is complete.
\end{proof}

\section{Appendix: convergence of the nonlinear term} \label{appendix-drift}

To simplify notations denote by, for $N\in \N$,
  $$\omega_N (x)= \Pi_N \omega (x)= \sum_{k\in \Lambda_N} \<\omega, \tilde e_k\> \tilde e_k(x),\quad x\in\T^2,$$
where the second equality is due to \eqref{sec-2.2}. It follows from \eqref{sec-2.1} that $\omega_N$ is a real-valued smooth function on $\T^2$. According to the Biot--Savart law,
  $$u(\omega_N)(x) = \int K(x-y) \omega_N(y)\,\d y= \pi {\rm i} \sum_{k\in \Lambda_N\setminus \{0\}} \<\omega, \tilde e_k\> \frac{k^\perp}{|k|^2} \tilde e_k(x),$$
which is a real, smooth divergence free vector field on $\T^2$. For any $\phi\in C^\infty(\T^2)$, by integration by parts,
  $$\aligned
  \<u(\omega_N) \cdot \nabla\omega_N, \phi\> &= -\<\omega_N, u(\omega_N) \cdot \nabla\phi\> = - \int \omega_N(x) (u(\omega_N) \cdot \nabla\phi)(x)\,\d x \\
  &= - \int\!\! \int \omega_N(x) \omega_N(y) K(x-y)\cdot \nabla\phi(x)\,\d x \d y.
  \endaligned$$
Using the property $K(x-y) = -K(y-x)$, the above equality can be rewritten as
  $$\<u(\omega_N) \cdot \nabla\omega_N, \phi\> = - \frac12 \int\!\! \int \omega_N(x) \omega_N(y) K(x-y)\cdot (\nabla\phi(x)- \nabla \phi(y))\,\d x \d y.$$
Denoting by
  \begin{equation}\label{eqn-1}
  H_\phi(x,y)= \frac12 K(x-y)\cdot (\nabla\phi(x)- \nabla \phi(y)),\quad (x,y)\in \T^2\times \T^2,
  \end{equation}
we obtain
  \begin{equation}\label{eqn-2}
  \<u(\omega_N) \cdot \nabla\omega_N, \phi\> = - \int\!\! \int \omega_N(x) \omega_N(y) H_\phi(x,y) \,\d x \d y = -\<\omega_N \otimes \omega_N, H_\phi \>.
  \end{equation}
Here $\omega_N \otimes \omega_N$ is a smooth function on $\T^2\times \T^2$ with the expression
  \begin{equation}\label{eqn-3}
  (\omega_N \otimes \omega_N)(x,y) = \sum_{k,l\in \Lambda_N} \<\omega, \tilde e_k\> \<\omega, \tilde e_l\> \tilde e_k(x) \tilde e_l(y).
  \end{equation}

\begin{lemma}\label{lem-app-1}
Assume that $f\in L^2(\T^2\times \T^2, \R)$ is a symmetric function. Then
  $$\big\{\<\omega_N \otimes \omega_N, f \>- \E_\mu\<\omega_N \otimes \omega_N, f \>\big\}_{N\in \N} $$
is a Cauchy sequence in $L^2(H^{-1-},\mu)$.
\end{lemma}

\begin{proof}
Denote by
  $$I_N:= \<\omega_N \otimes \omega_N, f \>,\quad N\in \N.$$
For $M,N \in \N$, $M>N$, define $\Lambda_{M,N} = (\Lambda_M \times \Lambda_M) \setminus (\Lambda_N \times \Lambda_N)$ which is a symmetric subset in $\Z^2_0 \times \Z^2_0$. Then by \eqref{eqn-3},
  \begin{equation}\label{lem-app-1.1}
  (\omega_M \otimes \omega_M - \omega_N \otimes \omega_N)(x,y) = \sum_{(k,l)\in \Lambda_{M,N}} \<\omega, \tilde e_k\> \<\omega, \tilde e_l\> \tilde e_k(x) \tilde e_l(y) =: \omega_{M,N}(x,y).
  \end{equation}
Note that $\omega_{M,N}(x,y)$ is real-valued. Now
  $$\aligned
  (I_M -I_N)^2 &= \<\omega_M \otimes \omega_M - \omega_N \otimes \omega_N, f \>^2 = \<\omega_{M,N} , f \>^2\\
  &= \bigg( \int\!\! \int \omega_{M,N}(x,y) f(x,y)\,\d x\d y \bigg)^2 \\
  &= \int\!\! \int \!\! \int \!\! \int f(x_1,y_1) f(x_2, y_2)\, \omega_{M,N}(x_1,y_1)\, \omega_{M,N}(x_2, y_2)\, \d x_1 \d y_1 \d x_2 \d y_2.
  \endaligned$$
We have, by \eqref{lem-app-1.1},
  $$\omega_{M,N}(x_1,y_1)\, \omega_{M,N}(x_2, y_2) = \sum_{(k,l), (i,j)\in \Lambda_{M,N}} \<\omega, \tilde e_k\> \<\omega, \tilde e_l\> \<\omega, \tilde e_i\> \<\omega, \tilde e_j\> \tilde e_k(x_1) \tilde e_l(y_1) \tilde e_i(x_2) \tilde e_j(y_2), $$
hence, denoting by $\E_\mu$ the expectation w.r.t. $\mu$ on $H^{-1-}$,
  \begin{equation}\label{lem-app-1.2}
  \aligned
  \E_\mu \big[ (I_M -I_N)^2 \big]= & \sum_{(k,l), (i,j)\in \Lambda_{M,N}} \E_\mu \big( \<\omega, \tilde e_k\> \<\omega, \tilde e_l\> \<\omega, \tilde e_i\> \<\omega, \tilde e_j\> \big)\\
  &\times\! \int\!\! \int \!\! \int \!\! \int \! f(x_1,y_1) f(x_2, y_2) \tilde e_k(x_1) \tilde e_l(y_1) \tilde e_i(x_2) \tilde e_j(y_2)\, \d x_1 \d y_1 \d x_2 \d y_2.
  \endaligned
  \end{equation}
Recall that
  \begin{equation}\label{prop-gradient-cylinder.3}
  \int \<\omega, \tilde e_p\> \<\omega, \tilde e_q\> \,\d\mu = \int \<\omega, \tilde e_p\> \overline{\<\omega, \tilde e_{-q} \>} \,\d\mu = \delta_{p,-q}, \quad p,q\in \Z^2.
  \end{equation}
By Isserlis--Wick theorem,
  $$\aligned
  \E_\mu \big( \<\omega, \tilde e_k\> \<\omega, \tilde e_l\> \<\omega, \tilde e_i\> \,  \<\omega, \tilde e_j\>\, \big) = & \ \E_\mu \big( \<\omega, \tilde e_k\> \<\omega, \tilde e_l\> \big) \, \E_\mu \big(\<\omega, \tilde e_i\> \<\omega, \tilde e_j\> \big) \\
  & + \E_\mu \big( \<\omega, \tilde e_k\> \<\omega, \tilde e_i\> \big)\, \E_\mu \big(\<\omega, \tilde e_l\> \<\omega, \tilde e_j\> \big) \\
  & + \E_\mu \big( \<\omega, \tilde e_k\> \<\omega, \tilde e_j\> \big) \, \E_\mu \big(\<\omega, \tilde e_l\> \<\omega, \tilde e_i\> \big) \\
  = &\ \delta_{k, -l} \delta_{i, -j} + \delta_{k, -i} \delta_{l, -j} + \delta_{k, -j} \delta_{l, -i} .
  \endaligned$$
Accordingly, the r.h.s. of \eqref{lem-app-1.2} is divided into three terms $J_1, J_2$ and $J_3$.

First, $J_1$ involves those terms such that $(k,-k),(i,-i)\in (\Lambda_M \times \Lambda_M) \setminus (\Lambda_N \times \Lambda_N)$, which implies $k,i \in \Lambda_M \setminus \Lambda_N$. Thus
  $$\aligned
  J_1 &= \sum_{k,i\in \Lambda_M \setminus \Lambda_N} \int\!\! \int \!\! \int \!\! \int f(x_1,y_1) f(x_2, y_2) \tilde e_k(x_1) \tilde e_{-k}(y_1) \tilde e_i(x_2) \tilde e_{-i}(y_2)\, \d x_1 \d y_1 \d x_2 \d y_2 \\
  & = \sum_{k,i\in \Lambda_M \setminus \Lambda_N} \int\!\! \int f(x_1,y_1) \tilde e_k(x_1) \tilde e_{-k}(y_1) \, \d x_1 \d y_1 \int\!\! \int f(x_2, y_2) \tilde e_i(x_2) \tilde e_{-i}(y_2)\, \d x_2 \d y_2\\
  & = \bigg(\sum_{k \in \Lambda_M \setminus \Lambda_N} \int\!\! \int f(x,y) \tilde e_k(x) \tilde e_{-k}(y) \, \d x \d y \bigg)^2 .
  \endaligned$$
Moreover, by \eqref{eqn-3} and \eqref{prop-gradient-cylinder.3},
  \begin{equation}\label{lem-app-1.2.5}
  \aligned
  \E_\mu I_N &= \sum_{k,l \in \Lambda_N} \E_\mu \big(\<\omega, \tilde e_k\> \<\omega, \tilde e_l\>\big) \int\!\! \int f(x,y) \tilde e_k(x) \tilde e_l(y)\,\d x\d y \\
  &= \sum_{k\in \Lambda_N} \int\!\! \int f(x,y) \tilde e_k(x) \tilde e_{-k}(y) \,\d x\d y.
  \endaligned
  \end{equation}
Therefore, we arrive at
  \begin{equation}\label{lem-app-1.3}
  J_1 = \big[\E_\mu (I_M-I_N) \big]^2.
  \end{equation}

Next, recall that $J_2$ corresponds to the case $k=-i$ and $l=-j$ in \eqref{lem-app-1.2}. Therefore,
  $$\aligned
  J_2 &= \sum_{(k,l)\in \Lambda_{M,N}} \int\!\! \int \!\! \int \!\! \int f(x_1,y_1) f(x_2, y_2) \tilde e_k(x_1) \tilde e_l(y_1) \tilde e_{-k}(x_2) \tilde e_{-l}(y_2)\, \d x_1 \d y_1 \d x_2 \d y_2 \\
  & = \sum_{(k,l)\in \Lambda_{M,N}}  \bigg|\int\!\! \int f(x,y) \tilde e_k(x) \tilde e_l(y) \, \d x \d y \bigg|^2 = \sum_{(k,l)\in \Lambda_{M,N}} |\<f,  \tilde e_k\otimes \tilde e_l\>|^2 .
  \endaligned$$
In the same way, since $f$ is symmetric,
  $$\aligned
  J_3 &= \sum_{(k,l)\in \Lambda_{M,N}} \int\!\! \int\!\! \int\!\! \int f(x_1,y_1) f(x_2, y_2) \tilde e_k(x_1) \tilde e_l(y_1) \tilde e_{-l}(x_2) \tilde e_{-k}(y_2)\, \d x_1 \d y_1 \d x_2 \d y_2 \\
  & = \sum_{(k,l)\in \Lambda_{M,N}}  \bigg|\int\!\! \int f(x,y) \tilde e_k(x) \tilde e_l(y) \, \d x \d y \bigg|^2 = \sum_{(k,l)\in \Lambda_{M,N}} |\<f,  \tilde e_k\otimes \tilde e_l\>|^2 .
  \endaligned$$

Combining the above computations with \eqref{lem-app-1.2} and \eqref{lem-app-1.3}, we obtain
  $$ \E_\mu \big[ (I_M -I_N)^2 \big] = \big[\E_\mu (I_M-I_N) \big]^2 + 2 \sum_{(k,l)\in \Lambda_{M,N}} |\<f, \tilde e_k\otimes \tilde e_l \>|^2.$$
This is equivalent to
  \begin{equation}\label{lem-app-1.4}
  \E_\mu \big[ (I_M - \E_\mu I_M) - (I_N - \E_\mu I_N)\big]^2 = 2\sum_{(k,l)\in \Lambda_{M,N}} |\<f,  \tilde e_k\otimes \tilde e_l\>|^2.
  \end{equation}
Since $f\in L^2(\T^2\times \T^2)$, the proof is complete.
\end{proof}

Using the Wiener chaos estimate (cf. \cite[Theorem I.22]{Simon} or \cite[Theorem 4.3]{DaPT}), we can greatly strengthen the above result.

\begin{proposition}\label{appendix-prop}
Let $f\in L^2(\T^2\times \T^2, \R)$ be a symmetric function. Then
  $$\big\{\<\omega_N \otimes \omega_N, f \>- \E_\mu\<\omega_N \otimes \omega_N, f \>\big\}_{N\in \N} $$
is a Cauchy sequence in $L^p(H^{-1-},\mu)$ for any $p\geq 1$.
\end{proposition}

\begin{proof}
The case $1\leq p\leq 2$ follows from Lemma \ref{lem-app-1}. Assume $p>2$ in the sequel. Note that
  $$\aligned
  I_N - \E_\mu I_N &= \sum_{k,l\in \Lambda_N, k\neq -l} \<\omega, \tilde e_k\> \<\omega, \tilde e_l\> \int\!\! \int f(x,y) \tilde e_k(x) \tilde e_l(y) \, \d x \d y\\
  &\hskip13pt + \sum_{k\in \Lambda_N} \big(\<\omega, \tilde e_k\> \<\omega, \tilde e_{-k}\> -1\big) \int\!\! \int f(x,y) \tilde e_k(x) \tilde e_{-k}(y) \, \d x \d y.
  \endaligned$$
We observe that the quantity $I_N - \E_\mu I_N$ belongs to the second Wiener chaos for any $N\geq 1$. Therefore, applying \cite[Theorem I.22]{Simon} with $m=2$ yields
  $$ \big\| (I_M - \E_\mu I_M) - (I_N - \E_\mu I_N)\big\|_{L^p(\mu)} \leq (p-1) \big\| (I_M - \E_\mu I_M) - (I_N - \E_\mu I_N)\big\|_{L^2(\mu)}. $$
This gives us the desired result.
\end{proof}

We need one more preparation.

\begin{lemma}\label{sec-5-lem}
For any fixed $l\in \Z^2_0$, one has
  $$\E_\mu \big\< \omega_N \otimes \omega_N, H_{e_l} \big\> =0 \quad \mbox{for all } N\in \N,$$
where $e_l$ is defined in \eqref{real-basis}.
\end{lemma}

\begin{proof}
By \eqref{eqn-1} and Lemma \ref{lem-2-1}, we have
  $$H_{e_l}(x,y) = \pi (e_{-l}(x) - e_{-l}(y))\, l \cdot K(x-y),\quad (x,y)\in \T^2 \times \T^2.$$
The identity \eqref{lem-app-1.2.5} leads to
  \begin{equation}\label{thm-app.3}
  \E_\mu \big\< \omega_N \otimes \omega_N, H_{e_l} \big\> = \pi \sum_{k\in \Lambda_N} l \cdot \int\!\! \int \tilde e_k(x) \tilde e_{-k}(y) (e_{-l}(x) - e_{-l}(y)) K(x-y)\,\d x\d y.
  \end{equation}
Assume $l\in \Z^2_+$, then $-l\in \Z^2_-$, hence by \eqref{real-basis.1} and \eqref{thm-app.4},
  $$\aligned
  \int\!\! \int \tilde e_k(x) \tilde e_{-k}(y) e_{-l}(x) K(x-y)\,\d x\d y &= \frac1{2{\rm i}} \int \tilde e_{k}(x) (\tilde e_{-l}(x) - \tilde e_l(x)) \,\d x \int \tilde e_{-k}(y) K(x-y)\,\d y\\
  & = - \pi \frac{k^\perp}{|k|^2} \int (\tilde e_{-l}(x) - \tilde e_l(x)) \,\d x =0
  \endaligned$$
as $l\neq 0 \in \Z^2_0$. Similarly, this integral vanishes when $l\in \Z^2_-$. In the same way,
  $$\int\!\! \int \tilde e_k(x) \tilde e_{-k}(y) e_{-l}(y) K(x-y)\,\d x\d y = 0.$$
Thus we obtain the desired result.
\end{proof}

Now we can prove the main results of this part.

\begin{theorem}\label{thm-app}
Let $G= g\circ \Pi_\Lambda$ be a cylindrical function for some $\Lambda \Subset \Z^2$ and $g\in C_P^\infty \big(\R^\Lambda \big)$. Then the series
  $$\big\{\big\< u(\omega_N) \cdot \nabla \omega_N , D G \big\> \big\}_{N\in \N}$$
converges in $L^p(H^{-1-}, \mu)$ for any $p\geq 1$.
\end{theorem}

\begin{proof}
We have, by \eqref{eqn-2},
  \begin{equation}\label{thm-app.0}
  D_N := \big\< u(\omega_N) \cdot \nabla \omega_N , D G \big\> = -\sum_{l\in \Lambda} \big(\partial_{\xi_l} g\big)(\Pi_\Lambda \omega) \big\< \omega_N \otimes \omega_N, H_{e_l} \big\> .
  \end{equation}
Therefore, for $M,N\in \N,\, M>N$,
  \begin{equation*}
  \aligned
  \E\big(|D_M - D_N |^p \big) &= \E \bigg( \bigg| \sum_{l\in \Lambda} \big(\partial_{\xi_l} g\big)(\Pi_\Lambda \omega) \big[\big\< \omega_M \otimes \omega_M, H_{e_l} \big\> - \big\< \omega_N \otimes \omega_N, H_{e_l} \big\>\big]\bigg|^p \bigg) \\
  & \leq C_{p, \Lambda} \sum_{l\in \Lambda} \E \Big(\big| \big(\partial_{\xi_l} g\big)(\Pi_\Lambda \omega) \big|^p \big|\big\< \omega_M \otimes \omega_M, H_{e_l} \big\> - \big\< \omega_N \otimes \omega_N, H_{e_l} \big\>\big|^p \Big)\\
  &\leq C_{p, \Lambda, g} \sum_{l\in \Lambda} \Big[\E \Big( \big|\big\< \omega_M \otimes \omega_M, H_{e_l} \big\> - \big\< \omega_N \otimes \omega_N, H_{e_l} \big\>\big|^{2p} \Big) \Big]^{1/2},
  \endaligned
  \end{equation*}
where in the last step we used Cauchy's inequality and the fact that $\partial_{\xi_l} g$ has polynomial growth for any $l\in \Lambda$. Applying Proposition \ref{appendix-prop} and Lemma \ref{sec-5-lem} yields the desired assertion.
\end{proof}

\section{Appendix: proof of the second assertion of Proposition \ref{prop-operator}} \label{appendix-diffusion}

Recall the definition of $R_{l,m}(N)$ in \eqref{prop-operator.2}. Since $l,m\in \Lambda$ is fixed, we shall simply write it as $R_N$.  Hence we consider
  \begin{equation}\label{first-limit}
  R_N= \sum_{|k|\leq N} C_{k,l} C_{k,m} \big(\<\omega, e_k e_l \> \< \omega, e_k e_m \> - \delta_{l,m} \big).
  \end{equation}
Note that, for all $N\geq 1$, $R_N$ belongs to the direct sum of the first two Wiener chaos, thus by \cite[Theorem I.22]{Simon}, it suffices to prove $R_N$ converges in $L^2(H^{-1-}, \mu)$.

\subsection{Case 1: $l\neq m$}

We have, for $0<N <M$,
  $$R_M- R_N= \sum_{N<|k| \leq M} C_{k,l} C_{k,m} \<\omega, e_k e_l \> \< \omega, e_k e_m \> $$
and
  \begin{equation*}
  \aligned
  \E_\mu \big[(R_M- R_N)^2 \big] = \sum_{N<|k|, |k'| \leq M} C_{k,l} C_{k,m} C_{k',l} C_{k',m} \E_\mu \big[\<\omega, e_k e_l \> \< \omega, e_k e_m \> \<\omega, e_{k'} e_l \> \< \omega, e_{k'} e_m \> \big].
  \endaligned
  \end{equation*}
By Isserlis-Wick theorem,
  $$\aligned
  &\hskip13pt \E_\mu \big[\<\omega, e_k e_l \> \< \omega, e_k e_m \> \<\omega, e_{k'} e_l \> \< \omega, e_{k'} e_m \> \big] \\
  &= \E_\mu\big[\<\omega, e_k e_l \> \<\omega, e_k e_m \> \big]\, \E_\mu\big[\<\omega, e_{k'} e_l \> \<\omega, e_{k'} e_m \> \big] \\
  &\hskip13pt + \E_\mu \big[\<\omega, e_k e_l \> \<\omega, e_{k'} e_l \> \big]\, \E_\mu \big[ \< \omega, e_k e_m \> \<\omega, e_{k'} e_m \> \big] \\
  &\hskip13pt + \E_\mu \big[\<\omega, e_k e_l \> \<\omega, e_{k'} e_m \> \big]\, \E_\mu \big[ \<\omega, e_k e_m \> \<\omega, e_{k'} e_l \> \big] .
  \endaligned$$
Accordingly, we can write
  \begin{equation}\label{eq-3}
  \E_\mu \big[(R_M- R_N)^2 \big]= S_1 + S_2 + S_3.
  \end{equation}

\subsubsection{The quantity $S_1$}

We have
  $$\aligned
  S_1 &= \sum_{N<|k|, |k'| \leq M} C_{k,l} C_{k,m} C_{k',l} C_{k',m} \E_\mu\big[\<\omega, e_k e_l \> \<\omega, e_k e_m \> \big]\, \E_\mu\big[\<\omega, e_{k'} e_l \> \<\omega, e_{k'} e_m \> \big] \\
  &= \bigg( \sum_{N<|k| \leq M} C_{k,l} C_{k,m} \E_\mu\big[\<\omega, e_k e_l \> \<\omega, e_k e_m \> \big] \bigg)^2.
  \endaligned$$
Note that
  $$\E_\mu\big[\<\omega, e_k e_l \> \<\omega, e_k e_m \> \big] = \int_{\T^2} e_k^2 e_l e_m \,\d x.$$
We deduce from \eqref{prop-gradient-cyl.3} and the fact $l\neq m$ that
  \begin{equation}\label{s-1}
  S_1=0.
  \end{equation}

\subsubsection{The quantity $S_2$} \label{sec-s-2}
By Cauchy's inequality,
  $$\aligned
  S_2&= \sum_{N<|k|, |k'| \leq M} C_{k,l} C_{k,m} C_{k',l} C_{k',m} \E_\mu \big[\<\omega, e_k e_l \> \<\omega, e_{k'} e_l \> \big]\, \E_\mu \big[ \< \omega, e_k e_m \> \<\omega, e_{k'} e_m \> \big]\\
  &\leq \bigg[\sum_{N<|k|, |k'| \leq M} C_{k,l}^2C_{k',l}^2 \Big(\E_\mu \big[\<\omega, e_k e_l \> \<\omega, e_{k'} e_l \> \big] \Big)^2 \bigg]^{1/2}\\
  &\hskip13pt \times \bigg[\sum_{N<|k|, |k'| \leq M} C_{k,m}^2C_{k',m}^2 \Big(\E_\mu \big[\<\omega, e_k e_m \> \<\omega, e_{k'} e_m \> \big] \Big)^2 \bigg]^{1/2}.
  \endaligned$$
We only consider the sum
  \begin{equation}\label{eq-s-2.1}
  \aligned
  J_{N,M} &= \sum_{N<|k|, |k'| \leq M} C_{k,l}^2 C_{k',l}^2 \Big(\E_\mu \big[\<\omega, e_k e_l \> \<\omega, e_{k'} e_l \> \big] \Big)^2\\
  &= \sum_{N<|k|, |k'| \leq M} C_{k,l}^2 C_{k',l}^2 \bigg(\int_{\T^2} e_k e_{k'} e_l^2 \,\d x\bigg)^2.
  \endaligned
  \end{equation}
Intuitively, this quantity tends to 0 as $M>N \to \infty$ due to the fact that the integral $\int_{\T^2} e_k e_{k'} e_l^2 \,\d x \neq 0$ imposes a constraint on $k$ and $k'$, e.g. $k=k'$ or $2l =k+k'$. Such constraint reduces the degree of freedom of $k$ and $k'$, and implies
  $$J_{N,M} \leq C_l \sum_{N<|k| \leq M} \frac1{|k|^4} \to 0 \quad \mbox{as } M>N \to \infty. $$

Below we present a more rigorous proof. Assume $l\in \Z^2_+$ for simplicity. Note that
  \begin{equation}\label{eq-5}
  \aligned
  \int_{\T^2} e_k e_{k'} e_l^2 \,\d x = \int_{\T^2} e_k e_{k'} (1+ \cos(4\pi l\cdot x)) \,\d x = \delta_{k,k'} + \int_{\T^2} e_k e_{k'} \cos(4\pi l\cdot x) \,\d x.
  \endaligned
  \end{equation}
If $k, k'\in \Z^2_+$, then
  \begin{equation*}
  \aligned
  e_k e_{k'} &= 2\cos(2\pi k\cdot x) \cos(2\pi k'\cdot x) = \cos(2\pi (k+k')\cdot x) + \cos(2\pi (k-k')\cdot x).
  \endaligned
  \end{equation*}
Note that $k+k'\in \Z^2_+$ while $k-k'$ may belong to either $\Z^2_+$ or $\Z^2_-$, even vanishes. Hence
  $$\aligned
  \int_{\T^2} e_k e_{k'} \cos(4\pi l\cdot x) \,\d x &= \int_{\T^2} \cos(2\pi (k+k')\cdot x) \cos(4\pi l\cdot x) \,\d x \\
  &\hskip13pt + \int_{\T^2} \cos(2\pi (k-k')\cdot x) \cos(4\pi l\cdot x) \,\d x \\
  &= \frac12 \big(\delta_{2l, k+ k'}+ \delta_{2l,\pm( k- k')} \big).
  \endaligned $$
Recall that $|k|, |k'|>N$, thus at most one of the norms $|k+k'|$ and $|k-k'|$ is less than $N$. Since $l\in \Z^2_+$ is fixed and $N\gg 1$, we conclude that at most one of the two quantities $\delta_{2l, k+ k'}$ and $\delta_{2l,\pm( k- k')}$ is nonzero. To summarize, we have obtained
  \begin{equation*}
  \int_{\T^2} e_k e_{k'} e_l^2 \,\d x = \delta_{k,k'} + \frac12 \big(\delta_{2l, k+ k'}+ \delta_{2l,\pm( k- k')} \big) .
  \end{equation*}
Therefore,
  \begin{equation}\label{eq-5.5}
  \bigg(\int_{\T^2} e_k e_{k'} e_l^2 \,\d x \bigg)^2= \delta_{k,k'} + \frac14 \big(\delta_{2l, k+ k'}+ \delta_{2l,\pm( k- k')} \big)
  \end{equation}
as all the cross products vanish.

If $k, k'\in \Z^2_-$, then
  \begin{equation*}
  \aligned
  e_k e_{k'} &= 2\sin(2\pi k\cdot x) \sin(2\pi k'\cdot x) = - \cos(2\pi (k+k')\cdot x) + \cos(2\pi (k-k')\cdot x) .
  \endaligned
  \end{equation*}
Thus we have
  $$\int_{\T^2} e_k e_{k'} \cos(4\pi l\cdot x) \,\d x = \frac12\big(- \delta_{2l, -(k+k')} + \delta_{2l, \pm(k-k')} \big) .$$
Combining this with \eqref{eq-5} leads to
  \begin{equation*}
  \int_{\T^2} e_k e_{k'} e_l^2 \,\d x= \delta_{k,k'} + \frac12\big(- \delta_{2l, -(k+k')} + \delta_{2l, \pm(k-k')} \big) .
  \end{equation*}
As a result,
  \begin{equation}\label{eq-7}
  \bigg(\int_{\T^2} e_k e_{k'} e_l^2 \,\d x \bigg)^2 = \delta_{k,k'} + \frac14\big( \delta_{2l, -(k+k')} + \delta_{2l, \pm(k-k')} \big) .
  \end{equation}

If $k\in \Z^2_+, k'\in \Z^2_-$, then
  \begin{equation*}
  \aligned
  e_k e_{k'} &= 2\cos(2\pi k\cdot x) \sin(2\pi k'\cdot x) = \sin(2\pi (k+k')\cdot x) - \sin(2\pi (k-k')\cdot x).
  \endaligned
  \end{equation*}
It is clear that
  $$\int_{\T^2} e_k e_{k'} e_l^2 \,\d x = \int_{\T^2} e_k e_{k'} \, (1+ \cos(4\pi l\cdot x) ) \,\d x = 0.$$
Similar result holds if $k\in \Z^2_-$ and $ k'\in \Z^2_+$.

Summarizing these arguments with \eqref{eq-s-2.1}, \eqref{eq-5.5} and \eqref{eq-7}, we obtain
  $$\aligned
  J_{N,M} &= \sum_{\substack{N<|k|, |k'| \leq M\\ k,k'\in \Z^2_+}} C_{k,l}^2 C_{k',l}^2 \bigg[ \delta_{k,k'} + \frac14 \big(\delta_{2l, k+ k'}+ \delta_{2l,\pm( k- k')} \big) \bigg] \\
  &\hskip13pt + \sum_{\substack{N<|k|, |k'| \leq M\\ k,k'\in \Z^2_-}} C_{k,l}^2 C_{k',l}^2 \bigg[ \delta_{k,k'} + \frac14 \big(\delta_{2l, -(k+k')}+ \delta_{2l,\pm( k- k')} \big) \bigg]
  \endaligned$$
It is obvious that the terms involving $\delta_{k,k'}$ tend to 0 as $M>N \to \infty$. Regarding the convergence of the other terms, we take the one containing $\delta_{2l, k+ k'}$ as an example. We have
  $$\aligned
  \sum_{\substack{N<|k|, |k'| \leq M\\ k,k'\in \Z^2_+}} C_{k,l}^2 C_{k',l}^2 \delta_{2l, k+ k'} &= \sum_{\substack{N<|k| \leq M\\ k\in \Z^2_+}} C_{k,l}^2 C_{2l-k,l}^2 {\bf 1}_{\{N<|2l -k| \leq M\}} {\bf 1}_{\{2l -k\in \Z^2_+\} } \\
  &\leq  \sum_{\substack{N<|k| \leq M\\ k\in \Z^2_+}} \frac{(k^\perp \cdot l)^2}{|k|^4} \frac{((2l-k)^\perp \cdot l)^2}{|2l -k|^4} \\
  &\leq C_l \sum_{\substack{N<|k| \leq M\\ k\in \Z^2_+}} \frac{1}{|k|^4} \to 0 \quad \mbox{as } M>N \to \infty.
  \endaligned$$

To summarize, we have proved
  \begin{equation}\label{s-2}
  \lim_{M>N \to \infty} S_2 =0.
  \end{equation}

\subsubsection{The quantity $S_3$} \label{sec-s-3}

Recall that
  $$S_3= \sum_{N<|k|, |k'| \leq M} C_{k,l} C_{k,m} C_{k',l} C_{k',m} \E_\mu \big[\<\omega, e_k e_l \> \<\omega, e_{k'} e_m \> \big]\, \E_\mu \big[ \<\omega, e_k e_m \> \<\omega, e_{k'} e_l \> \big] .$$
We have
  \begin{equation*}
  \E_\mu \big[\<\omega, e_k e_l \> \<\omega, e_{k'} e_m \> \big]= \E_\mu \big[ \<\omega, e_k e_m \> \<\omega, e_{k'} e_l \>\big] = \int_{\T^2} e_k e_l e_{k'} e_m \,\d x.
  \end{equation*}
Similar to the arguments in Section \ref{sec-s-2}, the fact that the integral $\int_{\T^2} e_k e_l e_{k'} e_m \,\d x \neq 0$ implies that $k$ and $k'$ are related to each other by some linear constraints, which in turn yield that $S_3$ tends to 0 as $M>N \to \infty$. We omit the details here.

Combining this with \eqref{eq-3}, \eqref{s-1} and \eqref{s-2} leads to
  $$\lim_{N,M\to \infty} \E_\mu \big[(R_M- R_N)^2 \big] =0.$$
Therefore, $R_N$ defined in \eqref{first-limit} is a Cauchy sequence in $L^2(\mu)$ in the case $l\neq m$.

\subsection{Case 2: $l= m$}

In this case, we have
  $$R_M- R_N= \sum_{N<|k| \leq M} C_{k,l}^2 \big(\<\omega, e_k e_l \>^2 -1 \big).$$
Therefore,
  \begin{equation}\label{case-2.0.1}
  \aligned &\hskip13pt \E_\mu \big[(R_M- R_N)^2 \big] = \sum_{N<|k|, |k'| \leq M} C_{k,l}^2 C_{k',l}^2 \E_\mu \big[ \big(\<\omega, e_k e_l \>^2 -1 \big) \big(\<\omega, e_{k'} e_l \>^2 -1 \big) \big]. \endaligned
  \end{equation}

We have
  \begin{equation}\label{case-2.0.0}
  \aligned &\hskip13pt \E_\mu \big[ (\<\omega, e_k e_l \>^2 -1 ) (\<\omega, e_{k'} e_l \>^2 -1 ) \big] \\
  &= \E_\mu \big[ \<\omega, e_k e_l \>^2 \<\omega, e_{k'} e_l \>^2 \big] - \E_\mu \big[ \<\omega, e_k e_l \>^2 \big] - \E_\mu \big[ \<\omega, e_{k'} e_l \>^2 \big] +1.
  \endaligned
  \end{equation}
Note that
  $$\aligned
  \sum_{N<|k|, |k'| \leq M} C_{k,l}^2 C_{k',l}^2 \E_\mu \big[ \<\omega, e_k e_l \>^2 \big] &= \bigg(\sum_{N<|k'| \leq M} C_{k',l}^2 \bigg) \sum_{N<|k| \leq M} C_{k,l}^2 \int_{\T^2} e_k^2(x) e_l^2(x) \,\d x.
  \endaligned$$
Similar to \eqref{cor-4}, it holds
  \begin{equation}\label{case-2.4.5}
  \sum_{N<|k| \leq M} C_{k,l}^2 e_k^2(x) = \sum_{N<|k| \leq M} C_{k,l}^2=: c_{M,N}.
  \end{equation}
Thus,
  \begin{equation}\label{case-2.5}
  \sum_{N<|k|, |k'| \leq M} C_{k,l}^2 C_{k',l}^2 \E_\mu \big[ \<\omega, e_k e_l \>^2 \big] = c_{M,N}^2 \int_{\T^2} e_l^2(x) \,\d x =c_{M,N}^2.
  \end{equation}
In the same way,
  \begin{equation}\label{case-2.6}
  \sum_{N<|k|, |k'| \leq M} C_{k,l}^2 C_{k',l}^2 \E_\mu \big[ \<\omega, e_{k'} e_l \>^2 \big] =c_{M,N}^2.
  \end{equation}

Next we deal with the first term in the second line of \eqref{case-2.0.0}.  By the Isserlis--Wick theorem,
  \begin{equation}\label{case-2.0}
  \aligned \E_\mu \big[ \<\omega, e_k e_l \>^2 \<\omega, e_{k'} e_l \>^2 \big] &= \E_\mu \big[ \<\omega, e_k e_l \>^2 \big] \E_\mu \big[ \<\omega, e_{k'} e_l \>^2 \big] \\
  &\hskip13pt + 2 \big( \E_\mu [ \<\omega, e_k e_l \> \<\omega, e_{k'} e_l \> ] \big)^2.
  \endaligned
  \end{equation}
We have, using \eqref{case-2.4.5},
  \begin{equation}\label{case-2.1}
  \aligned
  &\sum_{N<|k|, |k'| \leq M} C_{k,l}^2 C_{k',l}^2 \E_\mu \big[ \<\omega, e_k e_l \>^2 \big] \E_\mu \big[ \<\omega, e_{k'} e_l \>^2 \big] \\
  = &\, \bigg(\sum_{N<|k| \leq M} C_{k,l}^2 \E_\mu \big[ \<\omega, e_k e_l \>^2 \big] \bigg)^2 = c_{M,N}^2.
  \endaligned
  \end{equation}
Combining \eqref{case-2.5}--\eqref{case-2.1} with \eqref{case-2.0.1} and \eqref{case-2.0.0}, we obtain
  $$\E_\mu \big[(R_M- R_N)^2 \big] =  2 \sum_{N<|k|, |k'| \leq M} C_{k,l}^2 C_{k',l}^2 \big( \E_\mu [ \<\omega, e_k e_l \> \<\omega, e_{k'} e_l \> ] \big)^2.$$
Note that the r.h.s. is two times of $J_{M,N}$ defined in \eqref{eq-s-2.1}. Thus, using the results in Section \ref{sec-s-2}, we conclude that
  $$\aligned
  \E_\mu \big[(R_M- R_N)^2 \big] & \leq C_l \sum_{N<|k| \leq M} \frac{1}{|k|^4},
  \endaligned $$
which tends to 0 as $N,M \to \infty$.

\medskip

\noindent \textbf{Acknowledgement.} The second author is grateful to the financial supports of the National Natural Science Foundation of China (Nos. 11571347, 11688101), and the Special Talent Program of the Academy of Mathematics and Systems Science, Chinese Academy of Sciences.

\end{document}